\documentclass[11pt,a4paper,leqno]{amsart}

\usepackage{amsmath,amsfonts,amssymb,amsthm}
\usepackage[english]{babel}
\usepackage{xcolor}
\usepackage{datetime}
\usepackage{enumerate}
\usepackage[foot]{amsaddr}
\usepackage{tikz}
\usepackage{lipsum}
\usepackage{tikz-cd}
\usepackage{mathtools}
\usepackage{float}
\usepackage{xr}
\usepackage[colorlinks=true,linkcolor=blue]{hyperref}
\usepackage{amsmath}
\usepackage{theoremref}
\usepackage{forest}

\newcommand{\treee}[1]{\tikz[baseline=0.3ex]{
		\draw   (0, 0) node (a)  {$\bullet$}
		(0,-1) node (b)  {$#1$};
		\draw (b) -- (a);}}

 \newcommand{\treeC}[1]{\tikz[baseline=0.1ex]{
		\draw   (  0, 0) node (a)  {$\bullet$}
		(-.7,-1) node (b)  {$#1$}
		( .7,-1) node (c)  {$#1$};
		\draw (b) -- (a) -- (c);}}

\usepackage[symbol]{footmisc}

\newtheorem{theorem}{Theorem}

\newtheorem{lemma}{Lemma}

\newtheorem{corollary}{Corollary}

\newtheorem*{theorem*}{Theorem}

\newtheorem{theo}{Theorem}[section]
\newtheorem{coro}[theo]{Corollary}
\newtheorem{propo}[theo]{Proposition}

\newtheorem{theor}[theo]{Theorem}
\newtheorem{lem}[theo]{Lemma}

\theoremstyle{definition}
\newtheorem{remark}[theo]{Remark}

\numberwithin{equation}{section}

\def\finremark{\hfill$\clubsuit$}

\usepackage{xcolor}

\def\C{\mathbb{C}}
\def\D{\mathbb{D}}

\def\E{\mathbf{E}}
\def\V{\mathbf{V}}
\def\P{\mathbf{P}}
\def\E{\mathbf{E}}

\def\K{\mathcal{K }}

\title{The distance to the border of a random tree}

\author[V.\,J. Maci\'a]{V\'{\i}ctor J. Maci\'a}
\address[V\'{\i}ctor J. Maci\'a]{Departamento de Matem\'aticas, CUNEF Universidad, Spain.}
\email{victor.macia@cunef.edu}
\thanks{Research of V. Maci\'a was partially funded by grant MTM2017-85934-C3-2-P2 of Ministerio de Econom\'{\i}a y Competitividad of Spain and European Research Council Advanced Grant 834728}

\makeatletter
\@namedef{subjclassname@2020}{%
	\textup{2020} Mathematics Subject Classification}
\makeatother


\begin{document}
\raggedbottom

	\begin{abstract}
We consider a Galton–Watson process conditioned to have total progeny equal to \(n\) and investigate the asymptotic probability that the process exhibits a distance to the border (i.e., the minimal distance from the root to a leaf) of at least \(k\) as \(n\to\infty\). This work resolves an open question posed in \cite{Ara-Fer}. While similar asymptotic results appear in the literature on \(k\)-protected nodes (see \cite{Protection1-mean} and \cite{Protection-Clemens}), our approach is distinct, relying on Tauberian theorems and asymptotic results for large powers that have been established using Khinchin families. This methodology allows us to analyze the asymptotic behavior of the generating functions in a direct way, contrasting with the singularity analysis or combinatorial techniques typically used in the context of \(k\)-protected nodes.

	\end{abstract}

\maketitle
	\parskip=1.5mm

\setcounter{tocdepth}{3}

\normalsize

\normalsize

\section{Introduction}\label{sec: Intro}

For a rooted tree $T$ the distance to the border $\partial(T)$ is the minimum of the distances from the root to the leaves of $T$. The height, on the other hand, is the maximum of these distances.

In this paper, we are interested in obtaining asymptotic formulas, as the size $n$ tends to infinity, for the probability that a random rooted tree, conditioned to have size $n$, has distance to the border $\partial \geq k$.

More precisely, for a Galton-Watson process $T$ with offspring distribution given by the random variable $Y$, and conditioned to have total progeny $\#(T) = n$, we ask for the probability that this conditioned Galton-Watson process has distance to the border bigger or equal than $k \geq 0$:
$$(\star) \quad \P(\partial(T) \geq k \, | \, \#(T) = n).$$

This question but for the height is a classical one that has been dealt with by R\'enyi and Szekeres for Cayley trees, \cite{Renyi}, de Bruijn, Knuth, and Rice for plane trees, \cite{de Bruijn}, and Flajolet, Gao, Odlyzko, and Richmond for binary trees, \cite{FlajoletOdlyzko}.

Our interest in the distance to the border of a rooted tree comes from \cite{Ara-Fer}, in that paper just the case of Cayley trees and $k = 2,3$ are dealt with, and the general question is proposed, see \cite[p. 312]{Ara-Fer}. 

The main result of this paper is Theorem \ref{thm: mainthmprobabilistic}. There an asymptotic formula for the probability $(\star)$, which depends on some computable constants, is presented. 

The proof of Theorem \ref{thm: mainthmprobabilistic} is based on an iteration scheme involving Lagrange's equation. An asymptotic formula for the coefficients of the $k$-th iterate of that scheme is the content of Theorem \ref{thm:mainthmlagrange}.

In the combinatorial setting several authors have studied the distance to the border in the context of the so called $k$-protected nodes. In \cite{Protection1-mean} the authors calculate this same asymptotic for the probability $(\star)$, in the context of simply generated families of trees, and also the asymptotic for the mean distance to the border using singularity analysis. In \cite{Protection-Bona}, \cite{Protection-Gisang} and  \cite{Protection-Clemens} the authors calculate some parameters related with the distance to the border of a tree: \cite{Protection-Bona} gives an asymptotic formula for the probability that  a randomly selected node in a binary search tree of size $n$ its at distance to the border bigger or equal than $k$, as the size $n$ tends to infinity. The papers
\cite{Protection-Gisang} and \cite{Protection-Clemens} study the mean distance to the border for certain classes of rooted plane trees, as the size $n$ tends to infinity.

Our method is based on an iterative scheme that leverages Lagrange’s inversion formula and the Theory of Khinchin families. Using this method we can extend the main result of this paper, Theorem \ref{thm: mainthmprobabilistic}, to some limit cases, see Remark \ref{remark: extend} below. This framework connects the Theory of Khinchin families with the theory of simply generated random trees. By linking combinatorial constructions with probabilistic tools, our method opens up new avenues for analyzing local structural parameters in random trees. We place significant effort into explaining this connection in depth—even revisiting well-known constructions—to ensure that the paper is fully self-contained and accessible to readers from both combinatorial and probabilistic backgrounds.

Although Khinchin families appear implicitly in many papers on random trees, to the best of our knowledge, no one has previously observed the connection between Lagrange inversion and the reparametrization of these families. Through this connection, we derive classical results about random trees in a direct and elegant manner. Moreover, we obtain an explicit formula for the probability of extinction expressed as a series related to the solution of Lagrange’s equation. The connection with the theory of Khinchin families presented here and in \cite{K_dos} allows many classical results to be proved in a simple way. For further details on Khinchin families and their applications in combinatorics and probability, see \cite{K_exponentials}, \cite{CFFM3} and \cite{K_dos}.

\subsection{Some notations}

We denote $\D(0,R) \subset \C$ the open disk with center $z = 0$ and radius $R > 0$ and $\C$ the complex plane. $\D$ denotes the unit disk in the complex plane $\C$.

For a power series $A(z)=\sum_{n=0}^\infty a_n z^n$, we denote its $n$th coefficient by
$$a_n=\textsc{coeff}_{[n]}(A(z))\, , \quad \mbox{for each $n \ge 0$}\, .$$

For a random variable $X$ we denote $\E(X)$ and $\V(X)$ the expectation and variance of $X$, respectively.

Given two random variables $X$ and $Y$ we write $X\stackrel{d}{=}Y$ to signify equality in distribution.

For integers $q \geq 0$ and $m \in \{0,1,2,\dots,q-1\}$ we denote 
$$\mathbb{N}_{m,q} = \{n \in \mathbb{N} : n \equiv m, \mod q \}.$$

Given two sequences $a_n$ and $b_n$, with $b_n \not\equiv 0$, we write
\begin{align*}
	a_n \sim b_n, \quad \text{ as } n \rightarrow \infty \text{ and } n \in \mathbb{N}_{m,q}, 
\end{align*}
to mean that
\begin{align*}
	\lim_{\substack{n \to \infty; \\ n\in \mathbb{N}_{m,q}}}\frac{a_n}{b_n} = 1.
\end{align*}

\subsection*{Acknowledgements} The author would like to thank Prof.~Jos\'e L. Fern\'andez for bringing this problem to his attention and also for some enlightening discussions.

\section{Power series with positive coefficients}

\subsection{Khinchin families} The class $\K$ consists of the nonconstant power series $f(z) = \sum_{n \geq 0}a_n z^n$ with nonnegative coefficients and positive radius of convergence $R > 0$, and such that $f(0) =a_0> 0$.

To each power series $f$ in $\K$ we associate a one-parameter family of discrete random variables $(X_t)_{t \in [0,R)}$ taking values in $\{0,1, \ldots\}$. This family of discrete random variables is called the Khinchin family of $f$ and is given by
\begin{align*}
	\P(X_t = n) = \frac{a_n t^n}{f(t)}, \quad \text{ for all } n \geq 0 \text{ and } t \in (0,R).
\end{align*}
For $t=0$, we define $X_0 \equiv 0$. For each $t \in (0,R)$, the random variable $X_t$ is not a constant.

For  each $t \in [0,R)$, we denote
$\E(X_t) = m(t)$ and $\V(X_t) = \sigma^2(t)
$.
It turns out that
\begin{align*}
	m(t) = \frac{tf'(t)}{f(t)}, \quad \text{ and } \quad   \sigma^2(t) = tm'(t).
\end{align*}

The mean function $m(t)$ is an increasing diffeomorphism from $[0,R)$ onto its image; this follows from the fact that $\sigma^2(t) > 0$, for all $t \in (0,R)$.

Those $f\in \K$ such that  \begin{equation}\label{eq:req for existence of tau}\lim_{t \uparrow R} m(t)>1\end{equation} comprise a subclass $\K^\star$ that is quite relevant in what follows. If $f \in \K^\star$, the unique value $\tau \in(0, R)$  such that $m(\tau)=1$, is called the \textit{apex} of $f$. This apex $\tau$ is characterized by $\tau f^\prime(\tau)=f(\tau)$. Observe that the existence of apex for $f$ precludes $f$ from being a polynomial of degree 1, and, in particular, implies that $f^{\prime\prime}(\tau)>0$.

We refer to \cite{K_uno}, \cite{K_exponentials} and \cite{K_dos}, for a detailed exposition of the topic.

\medskip

To power series $g(z)$ of the form $g(z)=z^N f(z)=\sum_{n=N}^\infty b_n z^n$ with $f\in \K$ (and radius of convergence $R$) and integer $N \ge 1$ we also associate a family of random variables $(Y_t)_{t \in [0,R)}$. These $Y_t$ are given by
$$\P(Y_t=n)=\frac{b_n t^n}{g(t)}\, , \quad \mbox{for $t\in (0,R)$ and $n \ge N$}\,,$$
and $Y_0\equiv N$. Observe that the $Y_t$ take values in $\{N, N{+}1, \ldots\}$. If $(X_t)_{t \in [0,R)}$ is the Khinchin family of $f$ then
$$Y_t\overset{d}{=}X_t+N\,, \quad \mbox{for any $t\in[0,R)$}\, .$$
We refer to the family $(Y_t)_{t \in [0,R)}$ as a \textit{shifted} Khinchin family.

\medskip

For any $f(z) = \sum_{n = 0}^{\infty}a_n z^n$ in $\K$, we denote
\begin{align*}
	Q_{f} =\gcd\{n > 0 : a_n > 0\} = \lim_{N \rightarrow \infty} \gcd\left(\{0 < n \leq N: a_n \neq 0\}\right);
\end{align*}
observe that there is an \textit{auxiliary power series}  $h \in \K$ such that $f(z) = h(z^{Q_f})$.

\subsection{Meir-Moon Tauberian theorem.}

The following Tauberian theorem, due to Meir-Moon, see \cite{MeirMoon}, will be used later in the proof of  Theorem \ref{thm:mainthmlagrange}.

\begin{theorem}[Meir-Moon]\label{tauberian}
	Let $B,C$ be power series in $\K$ with Taylor expansions
	\begin{align*}
		B(z) = \sum_{n = 0}^{\infty}b_nz^n, \quad C(z) = \sum_{n=0}^{\infty}c_nz^n,
	\end{align*}
	and let
	\begin{align*}
		D(z) \triangleq B(z)C(z) = \sum_{n=0}^{\infty}d_nz^n.
	\end{align*}
	Assume that the coefficients of $B$ and $C$ satisfy
	\\
	\begin{align*}
		b_n = O(n^{-\beta} r^{-n}) \quad \text{ and } \quad c_n \sim Cn^{-\alpha} r^{-n},\quad \mbox{as $n \rightarrow \infty$},
	\end{align*}
	\\
	where $C,r,\alpha$ and $\beta$ are positive constants. If $\alpha <1 <\beta$ and $B(r) \neq 0$, then
	
	\begin{align*}
		d_n \sim B(r)c_n, \quad \text{ as } n \rightarrow \infty.
	\end{align*}
\end{theorem}

The Meir-Moon Theorem above is a companion of a result from Schur (and Szász), see \cite[p. 39]{PolyaSzego} and also \cite[Theorem VI.12, p. 434]{Flajolet}, which requires $B$ to have radius of convergence strictly bigger than $r > 0$ and also that $\lim_{n \rightarrow \infty}c_n/c_{n-1} = r$.

For similar results of this kind see \cite{MeirMoonOld}, \cite{MeirMoonChar} and \cite{MeirMoon}.

\begin{corollary}\label{tauberian-corollary-Q} Let $Q \geq 1$ be a positive integer. Let $B$ and $C$ be two power series in $\K$ with Taylor expansions
	\begin{align*}
		B(z) = \sum_{j = 0}^{\infty}b_{jQ}z^{jQ}, \quad C(z) = \sum_{j=0}^{\infty}c_{jQ}z^{jQ},
	\end{align*}
	and let
	\begin{align*}
		D(z) \triangleq B(z)C(z) = \sum_{j=0}^{\infty}d_{jQ}z^{jQ}.
	\end{align*}
	Assume that the coefficients of $B$ and $C$ satisfy
	\\
	\begin{align*}
		b_n = O(n^{-\beta} r^{-n}) \quad \text{ and } \quad c_n \sim Cn^{-\alpha} r^{-n},\quad \mbox{as $n \rightarrow \infty$ and $n \in \mathbb{N}_{0,Q}$},
	\end{align*}
	\\
	where $C,r,\alpha$ and $\beta$ are positive constants. If $\alpha <1 <\beta$ and $B(r) \neq 0$, then 
	
	\begin{align*}
		d_n \sim B(r)c_n, \quad \text{ as } n \rightarrow \infty \text{ and } n \in \mathbb{N}_{0,Q}
	\end{align*}
\end{corollary}

\section{Lagrange's equation}\label{sec:Lagrange}

Let $\psi$ be a power series in $\K^\star$ with radius of convergence $R_\psi > 0$. Let $g$ the  power series which is the (unique) solution of Lagrange's equation with data $\psi$, i.e. satisfying:
$$g(z)=z \psi(g(z)).$$

\subsection{Coefficients} The coefficients of the solution power series $g(z)=\sum_{n=0}^\infty A_n z^n$ are given in terms of the data $\psi$ by Lagrange's inversion formula:
$$A_n=\frac{1}{n} \textsc{coeff}_{[n{-}1]} \big(\psi(z)^n\big), \quad \mbox{for $n \ge 1$},$$
and $A_0=0$.

Besides the exact formula above for the $A_n$ in terms of the data $\psi$, there is the following asymptotic formula due to Otter, \cite{Otter} and to Meir-Moon, \cite{MeirMoonOld}. Recall that for any integer $m \in \{0,1,2,\dots,Q_{\psi}-1\}$ we denote $\mathbb{N}_{m,Q_\psi}=\{n \in \mathbb{N}: n \equiv m \, \textrm{ mod} \, Q_\psi\}$. Using this notation:

\begin{itemize}\item for indices $n \notin \mathbb{N}_{1,Q_{\psi}}$, we have that $A_n=0$,
	\item for indices $n \in \mathbb{N}_{1,Q_\psi}$, we have that
	\begin{equation}\label{eq:formula OtterMeinMoon}A_n \sim \frac{Q_\psi}{\sqrt{2\pi}} \frac{\tau}{\sigma_\psi(\tau)}\frac{1}{n^{3/2}} \Big(\frac{\psi(\tau)}{\tau}\Big)^n\, , \quad \mbox{as $n \to \infty$ and $n \in \mathbb{N}_{1,Q_{\psi}}$}\,.\end{equation}
\end{itemize}

The radius of convergence of the power series solution $g$ is denoted $\rho$. The asymptotic formula for the $A_n$ implies that $\rho=\tau/\psi(\tau)$ and, also,  that $g$ extends to be continuous in the closed disk $\{|z|\le \rho\}$, since the sequence $(A_n \rho^n)_{n \ge 1}$ is summable.

\subsection{Function $t/\psi(t)$}

Lagrange's equation gives that $g$ and $z/\psi(z)$ are inverse of each other and thus that
$$g(z/\psi(z))=z\, , \quad \mbox{for $|z|\le \tau$}\,.$$
In particular,
$$g(t/\psi(t))=t\, , \quad \mbox{for $t \in[0, \tau]$}\,.$$
The image of the interval $[0,\rho]$ by $g$ is the interval $g([0,\tau/\psi(\tau)]) = [0,\tau]$.

For the function $t \mapsto t/\psi(t)$, we have
\begin{lem}\label{lemma: monotony_t/psi(t)} For  $\psi \in \K^\star$ with apex $\tau \in (0,R_\psi)$, the function $t/\psi(t)$
	\begin{enumerate}
		\item is strictly increasing on the interval $[0,\tau)$
		\item is strictly decreasing on the interval $(\tau,R_{\psi})$
		\item has a maximum at $t = \tau$.
	\end{enumerate}
\end{lem}

\begin{proof}
	The result follows from the identity
	\begin{align*}
		\left(\frac{t}{\psi(t)}\right)^{\prime} = \frac{1}{\psi(t)}(1-m_{\psi}(t)), \quad \text{ for all } t \in [0,R_{\psi}).
	\end{align*}
\end{proof}

\subsection{Functions of the solution $g$ of Lagrange's equation}

For a power series $H$ with radius of convergence $R > \tau$ the coefficients of the power series $H(g)$ are given by
\begin{align}\label{eq:formula_H(g)}
	\textsc{coeff}_{[n]}\left(H(g(z))\right) = \frac{1}{n}\textsc{coeff}_{[n-1]}\left(H^{\prime}(z)\psi(z)^n\right), \text{ for } n \geq 1,
\end{align}
and $\textsc{coeff}_{[0]}\left(H(g(z))\right) = H(0)$.

For the coefficients of functions $H(g)$ of the solution $g$ we have the following asymptotic formula, see \cite[p. 268]{MeirMoonOld}.
\begin{lemma}
	\label{lemma:powers-Lagrange} Let $\psi \in \K^\star$ with apex $\tau$ and $Q_{\psi} = 1$. Let $g$ be the solution of Lagrange's equation with data $\psi$. Then, for any nonconstant power series  $H$, with nonnegative coefficients and having radius of convergence $R > \tau$, the following asymptotic formula holds
	\begin{align*}
		\textsc{coeff}_{[n]}\left[H(g(z))\right] \sim H^{\prime}(\tau)A_n, \quad \text{ as } n \rightarrow \infty.
	\end{align*}
\end{lemma}

Observe that $H^\prime(\tau)\neq 0$, since $H$ is nonconstant.

Now we generalize Lemma \ref{lemma:powers-Lagrange} to power series $\psi \in \K^{\star}$ with $Q_{\psi} \geq 1$. Before we state and prove this generalization we point out that for $\psi \in \K^{\star}$ with Khinchin family $(Y_t)$ and $Q_{\psi} \geq 1$ we have the equality in distribution
\begin{align}\label{eq:eq_distr_Q}
	Y_t \stackrel{d}{=} Q_{\psi} \cdot W_{t^Q}, \quad \text{ for all } t \in [0,R_{\psi}),
\end{align}
where $\psi(z) = \phi(z^{Q_{\psi}})$, for certain auxiliary power series $\phi \in \K$ with Khinchin family $(W_t)$.

From (\ref{eq:eq_distr_Q}) it follows that
\begin{align}\label{eq: mean_variance_Q}
	m_{\psi}(t) = Q_{\psi} \cdot m_{\phi}(t^Q) \quad \text{ and } \quad \sigma_{\psi}(t) = Q_{\psi}\cdot \sigma_{\phi}(t^{Q_{\psi}}),
\end{align}
for all $t \in [0,R_{\psi})$.

From (\ref{eq: mean_variance_Q}) it follows that
\begin{align}\label{eq:phi_tau_Q}
	m_{\phi}(\tau^{Q_{\psi}}) = 1/Q_{\psi}.
\end{align}

\begin{lem}\label{lemma:H_saltos} Let $\psi \in \K^\star$ with apex $\tau$ and $Q_{\psi} \geq 1$. Fix an integer $m \in \{0,1, \dots,Q_{\psi}~-~1\}$. Let $g$ be the solution of Lagrange's equation with data $\psi$. 
	
	Let $H$ be a nonconstant power series, with nonnegative coefficients and having radius of convergence $R > \tau$. Assume that 
	\begin{align*}
		\textsc{coeff}_{[n]}[H(z)] = 0, \quad \text{ if } n \notin \mathbb{N}_{m,Q_{\psi}},
	\end{align*}
	then 
	\begin{align*}
		\textsc{coeff}_{[n]}\left[ H(g(z))\right] \sim  H^{\prime}(\tau)\frac{Q_\psi}{\sqrt{2\pi}} \frac{\tau}{\sigma_\psi(\tau)}\frac{1}{n^{3/2}} \Big(\frac{\psi(\tau)}{\tau}\Big)^n\, , 
	\end{align*}
	$\mbox{as $n \to \infty$ and $n \in \mathbb{N}_{m,Q_{\psi}}$}$. 
	
	This implies that there exists a constant $C > 0$ such that 
	\begin{align*}
		\textsc{coeff}_{[n]}\left[H(g(z))\right] \leq C \, \rho^{-n} n^{-3/2},
	\end{align*}
	as $n \rightarrow \infty$ and $n \in \mathbb{N}_{m,Q_{\psi}}$.
	
\end{lem}
\begin{proof}
	By hypothesis $H(z) = z^m \hat{H}(z^{Q_{\psi}})$. Thus
	\begin{align}\label{eq: H_Q}
		H^{\prime}(z) = z^{m-1}\tilde{H}(z^{Q_{\psi}}), \quad \text{ for all } m \in \{1,2,\dots,Q_{\psi}-1\},
	\end{align}
	and $H^{\prime}(z) = z^{Q_{\psi}-1}{H}^{*}(z^{Q_{\psi}})$, for $m = 0$. Here $\hat{H}$, $\tilde{H}$ and $H^{*}$ are all power series with non-negative coefficients. 
	Besides we may write $\psi(z) = \phi\left(z^{Q_{\psi}}\right)$, where $\phi$ is a power series in $\K$.
	
	Formula (\ref{eq:formula_H(g)}) tell us that
	\begin{align*} 
		\textsc{coeff}_{[n]}\left[H(g(z))\right] = \frac{1}{n}\textsc{coeff}_{[n-1]}\left[H^{\prime}(z)\psi(z)^n\right],
	\end{align*}
	therefore, for integers $n \geq m$, using (\ref{eq: H_Q}), we find that
	\begin{align*}
		\textsc{coeff}_{[n]}\left[H(g(z))\right] = \frac{1}{n}\textsc{coeff}_{[n-m]}\left[\tilde{H}(z^{Q_{\psi}})\phi(z^{Q_{\psi}})^n\right],
	\end{align*}
	for all $m \in \{1,2,\dots,Q_{\psi}-1\}$ and also that
	\begin{align*}
		\textsc{coeff}_{[n]}\left[H(g(z))\right] = \frac{1}{n}\textsc{coeff}_{[n-Q_{\psi}]}\left[{H}^{*}(z^{Q_{\psi}})\phi(z^{Q_{\psi}})^n\right],
	\end{align*}
	for $m = 0$.
	
	Fix $l \in \{1,2,\dots,Q_{\psi}-1\}$. For $n = l+jQ_{\psi}$ we have
	\begin{align}\label{eq: H(g)_Q=1_pruebalemma}
		\frac{1}{n}\textsc{coeff}_{[n-l]}\left[\tilde{H}(z^{Q_{\psi}})\phi(z^{Q_{\psi}})^n\right] = \frac{1}{n}\textsc{coeff}_{[j]}\left[\tilde{H}(z)\phi(z)^n\right].
	\end{align}
	For $l = Q_{\psi}$ and $n = l+jQ_{\psi}$ we have
	\begin{align}\label{eq: H(g)_Q=1_pruebalemma2}
		\frac{1}{n}\textsc{coeff}_{[n-l]}\left[{H}^{*}(z^{Q_{\psi}})\phi(z^{Q_{\psi}})^n\right] = \frac{1}{n}\textsc{coeff}_{[j]}\left[{H}^{*}(z)\phi(z)^n\right].
	\end{align}
	
	From (\ref{eq: mean_variance_Q}) and (\ref{eq:phi_tau_Q}) it follows that
	\begin{align*}
		m_{\phi}(\tau^{Q_{\psi}}) = 1/Q_{\psi} \quad \text{ and } \quad \sigma_{\psi}(\tau) = Q_{\psi} \cdot \sigma_{\phi}(\tau^{Q_{\psi}}).
	\end{align*}
	Moreover $\tau^{Q_{\psi}j} = \tau^{n-m}$.
	
	Now we apply a large powers result from Section 8 in \cite{K_dos} (this is the case $j/n \rightarrow~1/Q_{\psi},$ as $n \rightarrow \infty$) to the right-hand side of (\ref{eq: H(g)_Q=1_pruebalemma}) and (\ref{eq: H(g)_Q=1_pruebalemma2}). Rearranging terms in the asymptotic formula (see Section 8 in \cite{K_dos}) we get that 
	\begin{align*}
		\textsc{coeff}_{[n]}\left[H(g(z))\right] \sim  H^{\prime}(\tau)\frac{Q_\psi}{\sqrt{2\pi}} \frac{\tau}{\sigma_\psi(\tau)}\frac{1}{n^{3/2}} \Big(\frac{\psi(\tau)}{\tau}\Big)^n\, , \quad \mbox{as $n \to \infty$} \text{ and } n \in \mathbb{N}_{m,Q_{\psi}}\,.
	\end{align*}
\end{proof}

\begin{remark}\label{remark: asymptotic_ineq} With the same hypothesis of Lemma \ref{lemma:H_saltos} we have that
	\begin{align*}
		\textsc{coeff}_{[n]}\left[ H(g(z))\right] \sim H^{\prime}(\tau) \rho^{-(m-1)}A_{n-m+1}, \quad \text{ as } n \rightarrow \infty \text{ and } n \in \mathbb{N}_{m,Q_{\psi}}.
	\end{align*}
	This follows combining Lemma \ref{lemma:H_saltos} with formula (\ref{eq:formula OtterMeinMoon}).
	
	From the previous asymptotic formula we obtain that there exists a constant $C_{H,\psi} > 0$ such that 
	\begin{align*}
		\textsc{coeff}_{[n]}\left[ H(g(z))\right] \leq C_{H,\psi}A_{n-m+1}, \quad \text{ as } n \rightarrow \infty \text{ and } n \in \mathbb{N}_{m,Q_{\psi}}.
	\end{align*}
	
	The previous relations also give that there exists a constant $C > 0$ such that
	\begin{align*}
		\textsc{coeff}_{[n]}\left[ H(g(z))\right] \leq C \rho^{-n}n^{-3/2}, \quad \text{ as } n \rightarrow \infty \text{ and } n \in \mathbb{N}_{m,Q_{\psi}}.
	\end{align*} \finremark
\end{remark}

\newpage

\section{Iteration and Lagrange's equation}

We fix now a power series $\psi(z) = \sum_{n =0}^{\infty}b_n z^n$ in $\mathcal{K}^{\star}$ with apex $\tau$. Recall that the radius of convergence of the solution $g$ of Lagrange's equation with data $\psi$ is given by $\rho=\tau/\psi(\tau)$.

\

For $z$ in the disk $\{|z| \leq \rho\}$ consider the sequence of power series given by the recurrence
\begin{align}\label{eq:iteration_scheme_height}
	g_n(z)=z \psi(g_{n{-}1}(z))\, , \quad \mbox{for $n \ge 1$}\, ,
\end{align}
starting off with $g_0(z)=b_0z$. 

This sequence is well defined. Observe that
\begin{align}\label{eq: ineq_iteration_altura}
	b_0 \, \rho = \tau/(\psi(\tau)/b_0) < \tau < R_{\psi}\, .
\end{align}

Using inequality (\ref{eq: ineq_iteration_altura}) combined with the fact that $\psi \in \K^{\star}$ we find that
\begin{align*}
	|g_1(z)| = |z\psi(b_0 z)| < \rho \, \psi(\tau)= \tau, \quad \text{ for all } |z| \leq \rho \, .
\end{align*}

By induction on $n \geq 0$ we conclude that for any integer $n \geq 0$ we have
\begin{align}\label{eq: g_n less than tau}
	|g_n(z)| < \tau, \quad \text{ for all } |z| \leq \rho\, ,
\end{align}
in particular, this implies that the functions in the sequence $\{g_n\}_{n \geq 0}$ are continuous on the boundary of the disk of center $z = 0$ and radius $\rho$. 

The sequence $(g_n)_{n \ge 1}$ converges uniformly in the closure of the disk $\D(0, \rho)$ to the power series $g$.
See Remark 2.7 in \cite[p. 9]{Sokal} and also \cite{Renyi} for more details about this iteration setting.

\

The iteration scheme we are interested in this paper is not the above but the following. Start with $g_0(z)=g(z)$ and define recursively
\begin{equation}\label{eq:iteration scheme} g_k(z)=z\big[\psi(g_{k{-}1}(z))-b_0\big], \quad \mbox{for $k \ge 1$}\, .\end{equation}

Write the expansion $g_k(z)=\sum_{n=1}^\infty A^{(k)}_n z^n$. It follows that $A_n^{(k)}=0$  for $1\le n \le k$ and that
\begin{equation}\label{eq:bounds for Ank}0 \le A_n^{(k)} \le A_n^{(k{-}1)}\le A_n, \quad \mbox{for $n \ge 1$ and $k \ge 1$}.\end{equation}

This means, in particular, that each power series $g_k$ has radius of convergence at least $\rho$, and that the sequence of the $g_k$ converges to $0$ uniformly in the closed disk $\textrm{cl}(\D(0, \rho))$

\

Define the bivariate analytic function
\begin{align*}
	G(z,w) = z(\psi(w)-b_0)\,, \quad \text{ for } |z| \leq \rho, |w| \leq \tau\,.
\end{align*}

The bound
\begin{align}\label{eq:Glesstau}
	|G(z,w)| \leq \tau\,, \quad \text{ for }  |z| \leq \rho, |w| \leq \tau\,,
\end{align}
follows since, for  $|z|\le\rho$ and $|w|\le \tau$, we have that
$$|z||\psi(w)-b_0|\le |z| \sum_{j=1}^\infty |b_j| |w|^j\le |z|\psi(|w|)\le\rho \psi(\tau)=\tau\,.$$

Inequality \eqref{eq:Glesstau} allows us to  iterate $G(z,w)$ with respect to the second variable: for any integer $k \geq 1$ we denote
\begin{align*}
	G_k(z,w) = G(z,G_{k-1}(z,w)), \quad \text{ for all } |z| \leq \rho \text{ and }  |w| \leq \tau\,,
\end{align*}
the $k$-th iterate of $G(z,w)$ in the second variable, starting with  $G_0(z,w) = w$.

For any integer $k \geq 0$ we have that
\begin{align}\label{eq:iteration_g_k}
	g_k(z) = G_k(z,g(z)), \quad \text{ for } |z| \leq \rho\,.
\end{align}

Observe that
\begin{equation}\label{eq:partial G}
	z \frac{\partial G}{\partial z}(z,w)=G(z,w) \quad \text{ and } \quad \frac{\partial G}{\partial w}(z,w)=z \psi^\prime(w), \quad \text{ for } |z| \leq \rho  \text{ and }  |w| \leq \tau\,.
\end{equation}

\

The main aim of this section is to obtain the following asymptotic result about the coefficients $A_n^{(k)}$:
\begin{theo}\label{thm:mainthmlagrange} For each $k \ge 0$,
	$$
	\lim_{\substack{n \to \infty; \\ n\in \mathbb{N}_{1,Q_\psi}}} \frac{A_n^{(k)}}{A_n}=\frac{\partial G_k}{\partial w}(\rho, \tau)\, .
	$$
\end{theo}

The proof of Theorem  \ref{thm:mainthmlagrange} is contained in Section \ref{section:proof_mainthmlagrange} and it is based on a number of lemmas about $G(z,w)$ which are described in the next section.

\begin{remark}\label{remark: extend}
	We can extend Theorem \ref{thm:mainthmlagrange} in the following way: assume that $\psi \in \K$ has radius of convergence $R_{\psi} < \infty$ and also that 
	\begin{enumerate}
		\item $\lim_{t \uparrow R_{\psi}}m_{\psi}(t) \triangleq M_{\psi} = 1$,
		\item $\lim_{t \uparrow R_{\psi}}\psi^{\prime\prime}(t) < \infty$.
	\end{enumerate}
	Denote $\rho = R_{\psi}/\psi(R_{\psi})$ and $\tau = R_{\psi}$, then we have
	\begin{align*}
		\lim_{\substack{n \to \infty; \\ n\in \mathbb{N}_{1,Q_\psi}}} \frac{A_n^{(k)}}{A_n}=\frac{\partial G_k}{\partial w}(\rho, \tau).
	\end{align*}
	
	The argument to prove this result is exactly the same than that of Theorem \ref{thm:mainthmlagrange}. The only difference appears in the proof of the auxiliary lemmas of Subsection \ref{subsec:concerningG(z,w)}, these proofs now rely on some asymptotic results appearing in \cite{K_dos}. These results extend the asymptotic formulas in Section \ref{sec:Lagrange} to the case $M_{\psi} = 1$ and $\tau = R_{\psi}$.
\finremark	
\end{remark}

\begin{remark}\label{remark generalized_iterate} It is possible to generalize the previous iteration scheme in the following way: consider $\psi(z) = \sum_{n \geq 0}b_n z^n \in \K^{\star}$ and fix $\mathcal{I} \subset \{n \geq 0 : b_n > 0\}$ a proper subset of indices such that $0 \in \mathcal{I}$. Define 
	\begin{align}\label{eq:generalized_iterate}
		S_{\mathcal{I}}(z,w) = z(\psi(w)-\psi_{\mathcal{I}}(w)) =z\psi_{\mathcal{I}^c}(w), \quad \text{ for all } |z| \leq \rho \text{ and } |w| \leq \tau\,,
	\end{align}
	where $\psi_{\mathcal{I}}(z) \triangleq \sum_{n \in \mathcal{I}}b_n z^n$ and $\mathcal{I}^c$ denotes the complement of $\mathcal{I}$ in the set of indices $\{n\geq 0 : b_n > 0\}$.
	
	Observe that
	\begin{align}\label{eq:iterate_section}
		|S_{\mathcal{I}}(z,w)| \leq G(|z|,|w|) \leq \tau, \quad \text{ for all } |z| \leq \rho \text{ and }|w| \leq \tau\,.
	\end{align}
	
	Inequality (\ref{eq:iterate_section}) allows us to iterate $S(z,w)$ with respect to the second variable: for any integer $m \geq 1$ we denote
	\begin{align*}
		S_{\mathcal{I},m}(z,w) = S_{\mathcal{I}}(z,S_{\mathcal{I},m-1}(z,w)), \quad \text{ for all } |z| \leq \rho \text{ and } |w| \leq \tau\,,
	\end{align*}
	the $m$-th iterate in the second variable of $S_{\mathcal{I}}(z,w)$, starting off with $S_{\mathcal{I},0}(z,w) = w$.
	
	Now consider the sequence of analytic functions 
	\begin{align}
		f_{\mathcal{I},m}(z) = S_{\mathcal{I},m}(z,g(z)), \quad \text{ for all } |z| \leq \rho \,,
	\end{align}
	this power series has Taylor expansion
	\begin{align}\label{eq:Taylor_f_I}
		f_{\mathcal{I},m}(z) = \sum_{n \geq 0}B_n^{(m)}z^n, \quad \text{ for all } |z| \leq \rho\,,
	\end{align} 
	and verifies the recurrence
	\begin{align*}
		f_{\mathcal{I},m}(z) = z(\psi-\psi_{\mathcal{I}})(f_{\mathcal{I},m-1}(z)),\quad \text{ for all } m \geq 1\,, 
	\end{align*}
	starting with $f_{\mathcal{I},0}(z) = g(z)$.
	
	Taking $\mathcal{I} = \{0\}$ in (\ref{eq:generalized_iterate}) we retrieve the iteration scheme (\ref{eq:iteration scheme}). 
\finremark	
\end{remark}

\subsection{Some lemmas concerning $G(z,w)$}\label{subsec:concerningG(z,w)}

The following lemma gives an explicit expression for the partial derivative with respect to $z$ of the iterate $G_k(z,w)$.

\begin{lem}\label{general-partial-z} For any integer $k \geq 1$ we have
	\begin{align}\label{partialzformula}
		\frac{\partial G_k}{\partial z}(z,w) &= \sum_{j=1}^{k}\frac{\partial G}{\partial z}(z,G_{k-j}(z,w))\prod_{i=1}^{j-1}\frac{\partial G}{\partial w}(z,G_{k-i}(z,w))\,,
	\end{align}
	for  $|z| \leq \rho$ and $|w| \leq \tau$, with the convention that the empty product is $1$.
\end{lem}

\begin{proof}
	Formula \eqref{partialzformula} is proved by induction on $k$. For $k = 1$, the right-hand side of \eqref{partialzformula} simplifies and reduces to
	$		({\partial G}/{\partial z})(z,w)$,
	as it should.
	
	Assume that  identity \eqref{partialzformula} holds for $k \ge 1$, then
	\begin{align*}
		(\star) \quad \frac{\partial G_{k+1} }{\partial z}(z,w) = \frac{\partial}{\partial z}G(z,G_k(z,w)) = \frac{\partial G}{\partial z}(z,G_k(z,w))+\frac{\partial G}{\partial w}(z,G_k(z,w))\frac{\partial G_k}{\partial z}(z,w)\,.
	\end{align*}
	Substituting the expression \eqref{partialzformula} for $\partial G_k/\partial z$  in $(\star)$ we find that
	\begin{align*}
		\frac{\partial G_{k+1}}{\partial z}(z,w) &= \sum_{j=1}^{k+1}\frac{\partial G}{\partial z}(z,G_{k+1-j}(z,w))\prod_{i=1}^{j-1}\frac{\partial G}{\partial w}(z,G_{k+1-i}(z,w))\,.
	\end{align*}
\end{proof}

The following lemma gives an asymptotic upper bound for the coefficients of $(\partial G_k/\partial z)(z,g(z))$.
\begin{lem}\label{lemma: bigOpartialz} For any integer $k \geq 1$ we have
	\begin{align*}
		\textsc{coeff}_{[n]}\left[\frac{\partial G_k}{\partial z}(z,g(z))\right] = O(\rho^{-n}n^{-3/2}), \quad \text{ as } n \rightarrow \infty \text{ and } n \in \mathbb{N}_{0,Q_{\psi}}\,.
	\end{align*}	
\end{lem}

\begin{proof} Combining formulas \eqref{eq:iteration_g_k} and \eqref{eq:partial G} with Lemma \ref{general-partial-z}
	we have that	\begin{align*}
		\textsc{coeff}_{[n]}\left[\frac{\partial G_k}{\partial z}(z,g(z))\right] &= \sum_{j = 1}^{k}\textsc{coeff}_{[n]}\left[ \frac{\partial G}{\partial z}(z,g_{k-j}(z))\prod_{i=1}^{j-1}\frac{\partial G}{\partial w}(z,g_{k-i}(z))\right] \\
		&=  \sum_{j = 1}^{k}\textsc{coeff}_{[n]}\left[ \frac{g_{k+1-j}(z)}{z}\prod_{i=1}^{j-1}\frac{\partial G}{\partial w}(z,g_{k-i}(z))\right]\,.
	\end{align*}
	
	From  the inequalities \eqref{eq:bounds for Ank} we then deduce, for $n \ge 1$, that
	\begin{align*}
		\sum_{j = 1}^{k}\textsc{coeff}_{[n]}\left[\frac{g_{k+1-j}(z)}{z}\prod_{i=1}^{j-1}\frac{\partial G}{\partial w}(z,g_{k-i}(z))\right]   &\leq \sum_{j = 1}^{k}\textsc{coeff}_{[n+1]}\left[g(z)\frac{\partial G}{\partial w}(z,g(z))^{j-1}\right]\,.
	\end{align*}
	
	For  $1 \le j \le k$ we have, using \eqref{eq:partial G}, that
	\begin{align*}
		\textsc{coeff}_{[n+1]}\left[g(z)\frac{\partial G}{\partial w}(z,g(z))^{j-1}\right] = \textsc{coeff}_{[n+2-j]}\left[g(z)\psi^{\prime}(g(z))^{j-1}\right]\,.
	\end{align*}
	
	Write $H(w)=w \psi^\prime(w)^{j{-}1}$. Observe that $H$ is a nonconstant power series which can be written as $H(w) = w^{1+(Q_{\psi}-1)(j-1)}\tilde{H}(z^{Q_{\psi}}) = w^{Q_{\psi}(j-1)+2-j}\tilde{H}(z^{Q_{\psi}})$, for certain auxiliary power series $\tilde{H}$\,.
	
	Applying Lemma \ref{lemma:H_saltos} with $H(w)=w \psi^\prime(w)^{j{-}1}$ and $m = Q_{\psi}(j-1)+2-j$, we conclude that there exists a constant $E_j > 0$ such that
	\begin{align*}
		\textsc{coeff}_{[n+2-j]}\left[g(z)\psi^{\prime}(g(z))^{j-1}\right] \leq E_j \, \rho^{-n}n^{-3/2}\,, \quad \text{ as } n \rightarrow \infty \text{ and } n\in \mathbb{N}_{0,Q_{\psi}}. \,
	\end{align*}
	Observe that for integers $n \geq 1$ such that $n \equiv 0,\text{\hspace{-.16 cm}}\mod Q_{\psi}$ we have that $n+2-j \equiv m$,$\mod Q_{\psi}$.
	
	Then we conclude, with $D_k=\sum_{j=1}^k E_j$, that
	\begin{align*}
		\textsc{coeff}_{[n]}\left[\frac{\partial G_k}{\partial z}(z,g(z))\right] \le D_k \, \rho^{-n}n^{-3/2}\,, \quad \text{ as } n \rightarrow \infty \text{ and } n\in \mathbb{N}_{0,Q_{\psi}}\,.
\end{align*}	\end{proof}

For the coefficients of $(\partial G_k/\partial w)(z,g(z))$ we have:

\begin{lem} \label{lemma: bigOpartial-w} For any integer $k \geq 1$ we have
	\begin{align*}
		\textsc{coeff}_{[n]}\left[\frac{\partial G_k}{\partial w}(z,g(z))\right] = O(\rho^{-n} n^{-3/2})\,, \quad \text{ as } n \rightarrow \infty \text{ and } n\in \mathbb{N}_{0,Q_{\psi}}\,.
	\end{align*}
\end{lem}

\begin{proof} 
	Applying the chain rule to $G_k(z,w)$  and using \eqref{eq:iteration_g_k} and \eqref{eq:partial G}, we obtain that
	\begin{align}\label{eq:formula for partial Gw}
		\frac{\partial G_k}{\partial w}(z,g(z)) = \prod_{j = 0}^{k-1}\frac{\partial G}{\partial w}(z,G_j(z,g(z))) = z^{k}\prod_{j = 0}^{k-1}\psi^\prime(G_j(z,g(z)))=z^k \prod_{j = 0}^{k-1}\psi^\prime(g_j(z))\,,
	\end{align}
	for  $|z| \leq \rho$ and $|w| \leq \tau$.
	
	Appealing to \eqref{eq:bounds for Ank} we deduce that
	$$		\textsc{coeff}_{[n]}\left[\frac{\partial G_k}{\partial w}(z,g(z))\right] \le
	\textsc{coeff}_{[n-k]}(\psi^\prime(g(z))^k)\,.$$
	
	First observe that for integers $j \geq 0$ we have
	\begin{align}\label{eq: ineqpsipsi_prime}
		\textsc{coeff}_{[j]}\left(\psi^{\prime}(g_j(z))\right) = \sum_{n \geq 0}b_n \textsc{coeff}_{[j]}\left(g_j(z)^n\right) \leq \textsc{coeff}_{[j]}\left[\psi^{\prime}(g(z))\right]\,.
	\end{align}
	Here we use the formula for the coefficients of a finite product of powers series and also that $A_n^{(l)} \leq A_n$, for any pair of nonnegative integers $l \geq 0$ and $n \geq 1$.
	
	Combining the convolution formula with inequality (\ref{eq: ineqpsipsi_prime}) we conclude that 
	\begin{align*}
		\textsc{coeff}_{[n]}\left[\frac{\partial G_k}{\partial w}(z,g(z))\right] &= \textsc{coeff}_{[n-k]}\left[\prod_{j = 0}^{k-1}\psi^\prime(G_j(z,g(z)))\right] \\ &\leq \textsc{coeff}_{[n-k]}\left[\psi^\prime(g(z))^{k}\right].
	\end{align*}
	
	Denote $H(w) = \psi^{\prime}(w)^k$, this function is nonconstant; recall that $\psi \in \K^{\star}$ cannot be a polynomial of degree $1$. The power series $H$ can be written as $H(w) = w^{(Q_{\psi}-1)k}\tilde{H}(w^{Q_{\psi}})$, for certain auxiliary power series $\tilde{H}$ with nonnegative coefficients.

	In fact $H(w) = w^{m}(w^{Q_{\psi}(j(Q_{\psi}-1)+l-1)}\tilde{H}(w^{Q_{\psi}}))$ with $m = Q_{\psi}-l$ and $k = jQ_{\psi}+l$, for $j \geq 0$ and $l \in \{0,1,\dots,Q_{\psi}-1\}$. Observe that $m+k \equiv 0$,$\mod Q_{\psi}$.
	
	Applying Lemma \ref{lemma:H_saltos} with $m$ as in the previous paragraph we find that
	\begin{align} \label{eq: asymptotic-partial-w}
		\textsc{coeff}_{[n-k]}[\psi^\prime(g(z))^k] \sim C_{H,\psi} \, \rho^{-n}n^{-3/2}\,, \quad \text{ as } n \rightarrow \infty \text{ and } n \in \mathbb{N}_{0,Q_{\psi}}\,,
	\end{align}
	where $C_{H,\psi}$ is a positive constant. For integers $n \geq k$ such that $n \equiv 0$,$\mod Q_{\psi}$ we have that $n-k \equiv m$,$\mod Q_{\psi}$.
	
	Using equation (\ref{eq: asymptotic-partial-w}), we conclude that there exists a constant $C > 0$ such that
	\begin{align*}
		\textsc{coeff}_{[n]}\left[\frac{\partial G_k}{\partial w}(z,g(z))\right] \leq C \, \rho^{-n} n^{-3/2}\,, \quad \text{ as } n \rightarrow \infty \text{ and } n \in \mathbb{N}_{0,Q_{\psi}}\,.
	\end{align*}
\end{proof}

\begin{lem}\label{lemma: partialwnonzero} For any $k \ge 0$ we have
	\begin{align*}
		\frac{\partial G_k}{\partial w}(\rho,\tau) \neq 0.
	\end{align*}
\end{lem}
\begin{proof} Using \eqref{eq:formula for partial Gw} and that $g(\rho)=\tau$, we have that
	$$\frac{\partial G_k}{\partial w}(\rho,\tau)=\frac{\partial G_k}{\partial w}(\rho,g(\rho))=\rho^k \prod_{j=0}^{k{-}1} \psi^\prime(g_j(\rho))\,.$$
	
	Now,  $\psi^\prime(g_j(\rho))=0$ could only happen if $\psi^{\prime}(0) = 0$ and $g_j(\rho)=0$, which means $g_j \equiv 0$. The iteration scheme \eqref{eq:iteration scheme} then implies that $g_{j-1}\equiv 0$, and then that $g_{j{-}2}\equiv 0$ and, finally, that $g \equiv 0$, which is not the case.
	Thus $(\partial G_k/\partial w)(\rho,\tau) \neq 0$.
	
\end{proof}

\subsection{Proof of Theorem \ref{thm:mainthmlagrange}.}\label{section:proof_mainthmlagrange}

Fix $k \ge 1$. From \eqref{eq:iteration_g_k} we have that
\begin{align*}
	g^{\prime}_k(z) = \frac{\partial G_k}{\partial z}(z,g(z))+\frac{\partial G_k}{\partial w}(z,g(z))g^{\prime}(z)\,, \quad \text{ for all } |z| < \rho\,.
\end{align*}
We estimate now the coefficients of $g_k^\prime(z)$ using the expression above.

Lemma \ref{lemma: bigOpartialz} tells us that the first summand of $g_k^\prime(z)$ satisfies
\begin{align*}
	\textsc{coeff}_{[n]}\left[\frac{\partial G_k}{\partial z}(z,g(z))\right] = O(\rho^{-n}n^{-3/2})\,, \quad \text{ as } n \rightarrow \infty \text{ and } n \in \mathbb{N}_{0,Q_{\psi}}\,.
\end{align*}


To estimate the coefficients of the second summand we apply Corollary \ref{tauberian-corollary-Q} with $(\partial G_k/\partial w)(z,g(z))$ in the role of $B(z)$ and $g^\prime(z)$ in the role of $C(z)$ and $r=\rho$.

For $(\partial G_k/\partial w)(z,g(z))$, using Lemma \ref{lemma: bigOpartial-w}, we have that
$$(\star)\quad \textsc{coeff}_{[n]}\left[\frac{\partial G_k}{\partial w}(z,g(z))\right]=O\big(\rho^{-n} n ^{-3/2}\big)\, ,\quad \mbox{as $n \to \infty$} \text{ and } n \in \mathbb{N}_{0,Q_{\psi}}\,.$$
On the other hand, for $g^\prime(z)$ and because of formula \eqref{eq:formula OtterMeinMoon} we have that
$$(\star\star)\quad
\textsc{coeff}_{[n]}\left(g^\prime(z)\right) \sim C \rho^{-n} n^{-1/2}\, , \quad \mbox{as $n \to \infty$} \text{ and } n \in \mathbb{N}_{0,Q_{\psi}}\,, $$ for a positive constant $C$.

The estimates $(\star)$ and $(\star\star)$ combined with Lemma \ref{lemma: partialwnonzero} allow us to apply Corollary \ref{tauberian-corollary-Q}, with $\alpha=1/2$ and $\beta=3/2$, and deduce that
\begin{align*}
	\textsc{coeff}_{[n]}\left[\frac{\partial G_k}{\partial w}(z,g(z))g^{\prime}(z)\right] \sim \frac{\partial G_k}{\partial w}(\rho,\tau)\, (n+1)A_{n{+}1}\,, \quad \text{ as } n \rightarrow \infty \text{ and } n \in \mathbb{N}_{0,Q_{\psi}}\,.
\end{align*}

%

Therefore we conclude that
\begin{align*}
	\textsc{coeff}_{[n]}\left[g_k^{\prime}(z)\right] = (n+1) A_{n+1}^{(k)} \sim \frac{\partial G_k}{\partial w}(\rho,\tau)\, (n+1) A_{n+1}\,, \quad \text{ as } n \rightarrow \infty \text{ and } n \in \mathbb{N}_{0,Q_{\psi}}\,,
\end{align*}
which is equivalent to
\begin{align*}
	\lim_{\substack{n \to \infty; \\ n\in \mathbb{N}_{1,Q_\psi}}} \frac{A_n^{(k)}}{A_n}=\frac{\partial G_k}{\partial w}(\rho, \tau)\, .
\end{align*}

\

\begin{remark}
	Assume that $Q_{\psi} = 1$. For any integer $m \geq 0$, we can apply the asymptotic estimates of this section to the coefficients of the sequence $f_{\mathcal{I},m}(z) = S_{\mathcal{I},m}(z,g(z))$. 
	
	The power series expansions of $S_{\mathcal{I},m}$ and $G_m$ are given by:
	\begin{align*}
		S_{\mathcal{I},m}(z,w) = \sum_{l,s \geq 0}S_{l,s}^{(m)}z^{l}w^m\quad \text{ and } \quad G_m(z,w) = \sum_{l,s \geq 0}G_{l,s}^{(m)}z^{l}w^m \,,
	\end{align*}
	for all $|z| \leq \rho$ and $|w| \leq \tau$\,.
	
	By induction on $m \geq 0$ it follows that
	\begin{align}\label{eq:ineq_S_m_G_m}
		S_{l,s}^{(m)} \leq G_{l,s}^{(m)}\,, \quad \text{ for all } l,s\geq 0\,. 
	\end{align}
	Combining inequality (\ref{eq:ineq_S_m_G_m}) with Lemmas \ref{lemma: bigOpartialz} and \ref{lemma: bigOpartial-w} we find that
	\begin{align*}
		\textsc{coeff}_{[n]}\left[\frac{\partial S_{\mathcal{I},m}}{\partial z}(z,g(z))\right] \leq \textsc{coeff}_{[n]}\left[\frac{\partial G_m}{\partial z}(z,g(z))\right]= O(\rho^{-n}n^{-3/2})\,, \quad \text{ as } n \rightarrow \infty\,,
	\end{align*}
	and also that
	\begin{align*}
		\textsc{coeff}_{[n]}\left[\frac{\partial S_{\mathcal{I},m}}{\partial w}(z,g(z))\right] \leq \textsc{coeff}_{[n]}\left[\frac{\partial G_m}{\partial w}(z,g(z))\right] = O(\rho^{-n}n^{-3/2})\,, \quad \text{ as } n \rightarrow \infty\,.
	\end{align*}
	
	Applying the same strategy than in the case $\mathcal{I} = \{0\}$ we find that
	\begin{align*}
		\lim_{n \rightarrow \infty}\frac{ \, \, B_n^{(m)}}{A_n} = \frac{\partial S_{\mathcal{I},m}}{\partial w}(\rho,\tau)\,.
	\end{align*}	
\finremark
\end{remark}

\section{Trees}

\subsection{Rooted trees}  A finite rooted tree $a$ is a finite connected graph without cycles with a distinguished node called root.

The set of leaves (nodes of degree 1, other than the root) of $a$ is  denoted by $\text{leaves}(a)$; we also refer to $\text{leaves}(a)$ as the \textit{border} of $a$.

\subsubsection{Generations and distances on a rooted tree}

If $u,v$ are distinct nodes on a rooted tree, its distance $d(u,v)$ is the length  (number of edges) of the unique path (not repeating nodes) that connects $u$ and $v$. If $u=v$,  the distance is $d(u,v)=0$.

For integer $n \geq 0$, the $n$th generation of a rooted tree $a$ is the set of nodes at exactly distance $n$ from the root.  The $0$-th generation consists simply of the root.

For a node $u$ in the $n$th generation of $a$, the neighboring nodes in generation $n+1$ are called the \textit{descendants} of $u$. The \textit{outdegree} of a node $u$ of $a$ is the number of its descendants. The leaves are the nodes with no descendants.
The \textit{genealogy} of a node $v$ of $a$ is the set of nodes in $a$ which lie in the (unique) path connecting the root to $v$.

%

The \textit{distance $\partial(a)$ to the border} of the rooted tree $a$ is defined as
\begin{align*}
	\partial(a) = \min\{d(\text{root},u): u \in \text{leaves}(a)\}\,,
\end{align*}
in other terms, the distance to the border of the rooted tree $a$ is the generation of the leaf with the shortest genealogy. In turn, the height $h(a)$ of a rooted tree is
$$h(a)=\max\{d(\text{root},u): u \in \text{leaves}(a)\}\,.$$

The \textit{outdegree profile} of a rooted tree $a$ is the list of nonnegative integers $$(k_0(a), k_1(a),\dots)\,,$$
where $k_j(a)$ is the number of nodes of $a$ with outdegree $j$. Observe that if $a$ has $n$ nodes then $k_j(a)=0$, for $j\ge n$, and that $\sum_{j = 0}^{n-1}k_j(a) = n$ and $\sum_{j = 1}^{n-1}jk_j(a) = n-1$.

\subsubsection{Weight of a tree associated to $\psi \in \K$} Fix $\psi(z) = \sum_{n = 0}^{\infty}b_n z^n \in \K$. For any rooted tree $a$ with $n$ nodes, the $\psi$-weight $w_\psi(a)$ of $a$ (associated to $\psi$) is
\begin{align}\label{eq: weight}
	w_{\psi}(a) = \prod_{j = 0}^{\infty}b_j^{k_j(a)},
\end{align}
where $(k_0(a), k_1(a),\dots)$ is the outdegree profile of  $a$.  The product above is actually a finite product, since $k_j(a)$ are not 0 only for a finite number of $j$'s.
For details about these weights we refer to \cite[p. 999]{MeirMoonOld}.
\subsubsection{Plane trees} A rooted plane tree is a finite rooted tree where each descendant of a given node has an intrinsic label which records its position within the tree, these labels follow a lexicographical order, see \cite{Neveu} and \cite[p. 126]{Flajolet} for more details. 

The previous definition is equivalent to specifying an order on the descendants of each node from left to right, this order allows us to embed our tree in the plane. 

We denote by $\mathcal{G}$ the class of finite rooted plane trees and for any integer $n \geq 1$ we denote by $\mathcal{G}_n$ the subclass of rooted plane trees with $n$ nodes. Further,  for any $k \geq 0$, we denote $\mathcal{G}^{(k)} \subseteq \mathcal{G}$ the class of rooted plane trees $a$ with $\partial(a) \geq k$ and $\mathcal{G}_n^{(k)} = \mathcal{G}^{(k)} \cap \mathcal{G}_n$, the class of rooted plane trees with $n \geq 1$ nodes and $\partial \geq k$.

We consider also infinite rooted plane trees which have infinitely many nodes, but each of them has a finite number of descendants. For the precise definition of this class we refer to \cite{Neveu}.

\subsection{Galton-Watson processes}\label{subsec: G-W-sinKhinchin} Let $Y$ be a discrete random variable with expectation $m\triangleq \E(Y) \leq 1$ and such that
\begin{enumerate}
	\item $\P(Y = 0) > 0$, \\[-0.4 cm]
	\item $\P(Y = 0)+\P(Y = 1) < 1$.
\end{enumerate}

Denote by $\psi$ the probability generating function of $Y$. Conditions (1) and (2) imply that $\psi(0) > 0$ and also that $\psi$ is not a polynomial of degree $1$.

We use $Y$ as the offspring distribution of a Galton-Watson stochastic process $(W_n)_{n\ge 0}$, with $W_n$ being the random size of the $n$-generation.

We start with a unique individual in generation zero: $W_0\equiv 1$. This individual has a number of offsprings with probability  distribution given by $Y$. These offsprings comprise the first generation whose size is $W_1$ (which has the same distribution as $Y$). Each individual of the first generation has a random number offsprings with probability distribution $Y$, independently  of each other; these offsprings comprise in turn the second generation of random size $W_2$. And so on.

For i.i.d copies $(Y_{i,j})_{i, j \ge 1}$ of the variable $Y$ and for any integer $n \geq 1$ we have
\begin{align*}
	W_n = Y_{n,1}+\dots+Y_{n,{W_{n-1}}}, \quad \mbox{for $n \ge 1$}\, .
\end{align*}
Besides $W_0 \equiv 1$.

The discrete process $(W_n)_{n \geq 0}$ is the Galton-Watson process (from now on abbreviated GW) with offspring distribution given by $Y$. For each integer $n \geq 0$, the random variable $W_n$ gives the number of individuals in the $n$th generation of our process.

Since $m=\E(Y)\le 1$, the GW process with offspring distribution $Y$ extinguishes with probability $1$, that is, $\P(W_m = 0, \text{ for some } m ) = 1$. See, for instance, \cite{Williams}.
The condition $m \le 1$, which here we assume from the start, is referred in the literature as the subcritical ($m<1$) and critical ($m=1$) cases.
The total progeny of the GW process is given by the random variable
\begin{align}\label{eq: total_progeny_NKhinchin}
	W \triangleq \sum_{j = 0}^{\infty}W_j \stackrel{d}{=} 1+\sum_{j = 1}^{\infty}W_j.
\end{align}
Since the process extinguishes with probability 1, the total progeny $W$ is finite almost surely.

We denote the probability generating function of $W$ by $g$. The power series $g$ verifies Lagrange's equation with data $\psi$, that is,
\begin{align*}
	g(z) = z\psi(g(z)), \quad \text{ for all } z \in \D.
\end{align*}
See, for instance, \cite{Pitman} and \cite{Williams}.

For further  basic background on classical Galton-Watson processes we refer, for instance, to \cite{Pitman} and \cite{Williams}.

\subsubsection{Galton-Watson processes and rooted plane trees} Assume that $Y$ is a discrete random variable with $m = \E(Y) \leq 1$. Each realization of a GW process with offspring distribution $Y$ can be identified with a finite rooted plane tree.

Galton-Watson processes are introduced as a model for the evolution of a population where each individual is distinguishable from the other. In this model we record the moment of birth and the relation parent/child of each individual. To take care of this information, we impose an order on the descendants of each node: we label intrinsically from left to right, as we do for the rooted plane trees, the position of the descendants of each node. 

To generate a Galton-Watson process we use independent copies of a discrete random variable $Y$. We start with one individual with offspring distribution $Y$. The descendants of this individual are ordered from left to right. Using this order we assign copies of $Y$, that we denote $Y_{1,1}, Y_{1,2},\dots, Y_{1, Y}$, to each of the new individuals in the first generation of our process, then we iterate this construction. 

This correspondence between nodes and random variables consider, tacitly, that the descendants of each node are ordered from left to right, otherwise, we could not assign a specific random variable to a specific individual and, therefore, we could not record the number of descendants of a particular individual. If nodes are not distinguishable, that is, if the descendants of each individual are not ordered, then, in general, we cannot register if a given node at the left or a given node at the right has certain number of descendants: we cannot distinguish left and right and hence we cannot distinguish different individuals. 

We can think about this (subcritical or critical) Galton-Watson process as a random variable $T$ which takes values in the set $\mathcal{G}$ of finite rooted plane trees.


For any integer $n \geq 1$ and any rooted plane tree $a \in \mathcal{G}_n$ the $\psi$-weight $w_\psi(a)$, here $\psi$ denotes the probability generating function of $Y$, is the probability of the event $\{T=a\}$.
\begin{align*}
	\P(T = a) = \prod_{j = 0}^{\infty}\P(Y = j)^{k_j(a)}=w_\psi(a),
\end{align*}
where $(k_0(a), k_1(a),\dots)$ is the outdegree profile of the tree $a$. 

Denote by $\#(T)$ the random variable that gives the total progeny of the GW process $T$, then we have $\#(T) \stackrel{d}{=} W$. Observe that $\P(\#(T) = n) = \P(T \in \mathcal{G}_n)$, for all $n \geq 1$.

We define the random variable \textit{distance to the border of the Galton-Watson process} $T$ as $\partial(T)$. The random variable $\partial(T)$ is given by
\begin{align*}
	\P(\partial(T) = k) = \P(T \in \mathcal{G}^{(k)} \setminus \mathcal{G}^{(k+1)})\,.
\end{align*}

We aim in this paper to obtain,  for any integer $k \geq 0$ and any GW process $T$,  an asymptotic formula for the conditional probability
\begin{align*}
	\P(\partial(T) \geq k \, | \, \#(T) = n) = \P(T \in \mathcal{G}^{(k)} \, | \, \#(T) = n), \quad \mbox{as $n \rightarrow \infty$.}
\end{align*}

\subsection{Parametric families of Galton-Watson processes}
Let $\psi(z) = \sum_{n = 0}^{\infty}b_n z^n$ be a power series in $\K^\star$ with radius of convergence $R_\psi$ and apex $\tau \in (0, R_\psi)$, and denote by $(Y_t)_{t \in [0,R_{\psi})}$ its Khinchin family.

We let $g(z) = \sum_{n \geq 1}A_n z^n$ be the power series solution of Lagrange's equation with data $\psi$ which has radius of convergence $\rho=\tau/\psi(\tau)$, and denote by $(Z_t)_{t \in [0,\tau/\psi(\tau))}$ its associated  shifted Khinchin family. Observe that $Z_0 \equiv 1$.

\medskip

We consider the one-parameter family of GW processes $(T_t)_{t \in [0,R_{\psi})}$ where   $T_t$ is the GW process with offspring distribution given by $Y_t$. For $t=0$, the GW process degenerates, its total progeny is just  one individual (or in terms of trees, just the root).

For each $t \in [0,R_{\psi})$ we denote by $q(t)$, the probability of extinction for the GW process $T_t$. We have $q(t)=1$, for $t \in [0, \tau]$, since for $t$ in that interval,  $m_\psi(t)\le 1$. For $t \in (\tau,R_{\psi})$, we have that $0 <q(t)< 1$, see, for instance, \cite[p. 4]{Williams}. 

\medskip

For each $t \in (0,R_{\psi})$, we let $\psi_t(z) = \psi(tz)/\psi(t)$ be the probability generating function of the random variable $Y_t$; we set $\psi_0\equiv1$.

We consider, for each $t \in [0,R_{\psi})$, Lagrange's equation with data $\psi_t$ and the corresponding  power series solution $g_t(z)$:
\begin{align}\label{eq: Lagrange-probabilitygen}
	g_t(z) = z\psi_t(g_t(z)).
\end{align}

The case $t = 0$ is a degenerate case: $\psi_0(z) \equiv 1$ and $g_0(z) = z$.

\medskip

\begin{propo}\label{lemma: g_t-probgen} For $t \in [0,R_\psi)$ we have that
	\begin{align}\label{eq:identity for gt}
		g_t(z) = \frac{g(tz/\psi(t))}{t}, \quad \mbox{for all $t \in [0,R_{\psi})$ and $|z| \leq 1$}\,.
	\end{align}	
	The holomorphic function  $g_t$ is continuous on $\partial \D$ and $g_t(\D)\subset \D$.
	
	Besides, for $t \in [0,\tau]$, the power series $g_t(z)$ is the probability generating function of the random variable $Z_{t/\psi(t)}$.
\end{propo}

\begin{proof} We write $h_t(z)= {g(tz/\psi(t))}/{t}$. Using that $g$ satisfies Lagrange's equation with data $\psi$, it is immediate to verify that $h_t$ satisfies Lagrange's equation with data $\psi_t$. Uniqueness of the solution gives that $g_t\equiv h_t$.
	
	Equation \eqref{eq:identity for gt} means that
	\begin{equation}\label{eq:power series for gt}g_t(z)=\sum_{n \ge 1} \frac{A_n t^{n{-}1}}{\psi(t)^n} z^n\, . \end{equation}
	
	The Otter-Meir-Moon formula \eqref{eq:formula OtterMeinMoon} and formula \eqref{eq:identity for gt} give that the radius of convergence of the power series $g_t$ is
	$$R_{g_t}=\frac{\tau/\psi(\tau)}{t/\psi(t)}\,,$$ and also the continuity of $g_t$ on $\partial \D$. Observe that
	$R_{g_t}>1$ for $t \neq \tau$, while $R_{g_\tau}=1$, see Lemma \ref{lemma: monotony_t/psi(t)}.
	
	Equation \eqref{eq: Lagrange-probabilitygen} is then satisfied for $|z|\le 1$, for all $t \in [0,R_{\psi})$.
	\medskip
	
	Now, for $t \in [0,\tau]$, since then $g(t/\psi(t))=t$, we deduce  that
	$$\begin{aligned}
		\P(Z_{t/\psi(t)}=n)&=\frac{A_n (t/\psi(t))^n}{g(t/\psi(t))}\\&=\frac{A_n (t/\psi(t))^n}{t}=\textsc{coeff}_{[n]}(g_t(z))\, , \quad \mbox{for $t \in [0,\tau]$ and $n \ge 1$}\, .\end{aligned}$$
	Therefore, as claimed,  $g_t$ is the probability generating function of $Z_{t/\psi(t)}$, for all $t \in [0, \tau]$.
\end{proof}

\begin{remark}
	The case $t = 0$ of this result needs of the convention $0^0 = 1$.  \finremark
\end{remark}

\medskip

Taking $z = 1$ in formula \eqref{eq: Lagrange-probabilitygen} we find that $g_t(1)$ is a fixed point for $\psi_t(z)$, that is,
\begin{align*}
	\psi_t(g_t(1)) = g_t(1), \quad \text{ for all }	t \in [0,R_{\psi}).
\end{align*}

From here we conclude that the probability of extinction $q(t)$ is given by
\begin{align}\label{eq: extinction}
	q(t) = g_t(1), \quad \text{ for all } t \in [0,R_{\psi}).
\end{align}
See \cite[p. 4]{Williams} for more details about this equality.

From formula (\ref{eq: Lagrange-probabilitygen}) we deduce that the probability of extinction function $q(t)$ maybe be expressed as
\begin{align}\label{eq: formula_extinction}
	q(t) = \sum_{n = 1}^{\infty}\frac{A_n t^{n-1}}{\psi(t)^n}, \quad \text{ for all } t \in [0,R_{\psi}),
\end{align}
with the convention that $0^0 = 1$. Recall that $q(t)=1$, for $t \in [0,\tau)$.

\medskip

\begin{coro}\label{lemma: progenydistribution} For each $t \in [0,\tau]$,  the power series $g_t$ is the probability generating function of $\# (T_t) $ and the following equality in distribution holds
	\begin{align*}
		\# (T_t) \stackrel{d}{=} Z_{t/\psi(t)}.
	\end{align*}
\end{coro}

\begin{proof}
	For all $t \in [0,\tau]$, the power series $g_t$ is a probability generating function. Since $g_t$ is the solution of  \eqref{eq: Lagrange-probabilitygen} and $\psi_t$ is the probability generating function of $Y_t$, we have that $g_t$ is the probability generating function of the size of the progeny $\# (T_t)$.
\end{proof}

\begin{remark} For each $t \in (\tau,R_{\psi})$, $\#(T_t)$ is a proper random variable. In this case, each $\#(T_t)$ takes values in $\{1,2,3,\dots\} \cup \{\infty\}$, therefore $(\#(T_t))_{t \in [0,R_{\psi})}$ is a, proper, one-parameter family of random variables.
\finremark
\end{remark}

\subsection{The conditional random tree}

Let $\mathcal{R} \subset \mathcal{G}$ be a subclass of rooted plane trees.  For integers $n \ge 1$, denote by $\mathcal{R}_n$ the subclass of trees of $\mathcal{R}$ of size $n$.

Fix $\psi(z) = \sum_{n = 0}^{\infty}b_n z^n$ in $\K^\star$ with radius of convergence $R_\psi$ and apex $\tau \in (0,R_{\psi})$.
For  integer $n \ge 1$, we denote with $U_n$ the sum of the $\psi$-weights of the trees in the class $\mathcal{G}_n$ and, similarly, $V_n$ the sum of the $\psi$-weights of the trees in the class $\mathcal{R}_n$:
\begin{align}\label{eq: sum_weigh_partial_root}
	U_n = \sum_{a \in \mathcal{G}_n}w_\psi(a)\,, \quad \text{ and } \quad V_n = \sum_{a \in \mathcal{R}_n}w_\psi(a)\,,
\end{align}
see equation (\ref{eq: weight}) and also \cite{MeirMoonOld}.

Fix $t \in (0,\tau]$. For the GW process $T_t$ and any rooted plane tree $a \in \mathcal{G}_n$, the probability that $T_t$ is equal to $a$ is given by
\begin{equation}\label{eq: prob_GW_finito}
	\P(T_t = a) = \prod_{j = 0}^{n-1}\P(Y_t = j)^{k_j(a)} =\omega_{\psi_t}(a)=\prod_{j = 0}^{n-1}\frac{b_j^{k_j}t^{j k_j}}{\psi(t)^{k_j}} =w_\psi(a)\frac{t^{n-1}}{\psi(t)^n}\,.
\end{equation}
See, for instance, \cite{Neveu} or \cite{Otter}.

Thus,
\begin{equation}\label{eq: prob_GW_finito subclass}
	\P(T_t \in \mathcal{R}_n) = V_n\frac{t^{n-1}}{\psi(t)^n}\,,
\end{equation}
and, in particular,
\begin{equation}\label{eq: prob_GW_finito class}
	\P(T_t \in \mathcal{G}_n) = U_n\frac{t^{n-1}}{\psi(t)^n}\,,
\end{equation}

The next proposition follows from \eqref{eq: prob_GW_finito subclass} and \eqref{eq: prob_GW_finito class}.

\begin{propo}\label{thm: uniform} For $t \in (0,R_{\psi})$,  we have
	\begin{align*}
		\P(T_t \in \mathcal{R}  \, | \, \#(T_t) = n) = \frac{V_n}{U_n}\,,
	\end{align*}
	for any integer $n \geq 1$ such that $n \equiv 1,\text{\hspace{-.16 cm}}\mod Q_{\psi}$.
\end{propo}

\begin{proof} For integers $n \equiv 1,\mod Q_{\psi}$, just observe that
	$$\P(T_t \in \mathcal{R}  \, | \, \#(T_t) = n)=\frac{\P(T_t \in \mathcal{R}_n)}{\P(T_t \in \mathcal{G}_n)}$$
	and use \eqref{eq: prob_GW_finito subclass} and \eqref{eq: prob_GW_finito class}.\end{proof}

\medskip

For $t \in (0, R_{\psi})$ and $n \ge 1$, we have that
$$\P(T_t \in \mathcal{G}_n)=\P(\# (T_t)=n)=\textsc{coeff}_{[n]}(g_t(z ))$$
and thus from formula \eqref{eq:power series for gt} we see that
$$\P(T_t \in \mathcal{G}_n)=\frac{A_n t^{n{-}1}}{\psi(t)^n}\,.$$
Comparing with \eqref{eq: prob_GW_finito class} we see that
$$A_n=U_n=\sum_{a \in \mathcal{G}_n}\omega_\psi(a)\,, \quad \mbox{for $n \ge 1$}\,.$$

\section{The distance to the border}\label{sec: distance_border}

Now we introduce a formula for the probability that a Galton-Watson process $T_t$, conditioned to have exactly $n$ nodes, has $\partial(T_t) \geq k$.

We recall here the notation $\mathcal{G}$ for the class of finite rooted plane trees and $\mathcal{G}^{(k)} \subseteq \mathcal{G}$ for the subclass of finite rooted plane trees with $\partial \geq k$.

For any $\psi \in \K$, we denote
\begin{align}\label{eq: sum_weigh_partial_root}
	A_n = \sum_{a \in \mathcal{G}_n}w(a), \quad \text{ and } \quad A_n^{(k)} = \sum_{a \in \mathcal{G}_n^{(k)}}w(a),
\end{align}
here $w(a)$ denotes the weight of $a$ with respect to a given $\psi \in \K$, see equation (\ref{eq: weight}).

For any integer $k \geq 1$ and any $t \in [0,\tau]$, the expression $A_n^{(k)}t^{n-1}/\psi(t)^n$ does not define a probability distribution. In particular, this expression does not define a Khinchin family: 
\begin{align*}
	\P(T_t \in \mathcal{G}^{(k)}) = \sum_{n \geq 1} \, \P(T_t \in \mathcal{G}_n^{(k)}) = \sum_{n \geq k} \, \frac{A_n^{(k)}t^{n-1}}{\psi(t)^n} <  \sum_{n \geq 1}\, \frac{A_n t^{n-1}}{\psi(t)^n} = 1.
\end{align*}
Here we use that for all $k \geq 1$, the class of rooted trees $\mathcal{G}_n^{(k)}$ is strictly contained in the class $\mathcal{G}_n$.

On the other hand for $k = 0$, we have the equality $A_n^{(0)} = A_n$, for all $n \geq 1$, therefore $\P(T_t \in \mathcal{G}^{(0)}) = \P(T_t \in \mathcal{G}) = 1$, for all $t \in [0,\tau]$.

In the previous reasoning we have used that, for any integer $k \geq 0$, we can partition the class of rooted plane trees with $\partial \geq k$ by its number of nodes:
\begin{align*}
	\mathcal{G}^{(k)} = \bigsqcup_{n = k}^{\infty} \mathcal{G}_n^{(k)} \subseteq \mathcal{G},
\end{align*}
where $\mathcal{G}_n^{(k)} \cap \mathcal{G}_m^{(k)} = \emptyset$, for any pair of integers $n,m \geq 1$ such that $n \neq m$.

\begin{lem}\label{lemma:uniform_galton} Fix $\psi \in \K$. For any integer $k \geq 0$, and any $t \in (0,R_{\psi})$, we have
	\begin{align*}
		\P(\partial(T_t) \geq k \, | \, \#(T_t) = n) = \P(T_t \in \mathcal{G}^{(k)}  \, | \, \#(T_t) = n) =  \frac{\, \, \,  A_n^{(k)}}{A_n},
	\end{align*}
	for all $n \in \mathbb{N}_{1,Q_{\psi}}.$
\end{lem}

The proof of this Lemma follows applying Theorem \ref{thm: uniform} with $\mathcal{R}_n = \mathcal{G}^{(k)}_n$.

\subsection{Asymptotic results}

Now we give a recurrence relation for the analytic function with coefficients $A_n^{(k)}$. Using this representation we give an asymptotic formula for the sequence $A_n^{(k)}$, as the number of nodes $n \rightarrow \infty$.

For any integer $k \geq 0$, we denote by $g_k(z)$ the analytic function with coefficients $A_n^{(k)}$, that is,
\begin{align*}
	g_k(z) = \sum_{n = k}^{\infty}A_n^{(k)}z^{n}, \quad \text{ for all } |z|\leq \rho \triangleq \tau/\psi(\tau).
\end{align*}

For all $t \in [0,R_{\psi})$ and $|z| \leq 1$, we have
\begin{align*}
	\frac{g_k(tz/\psi(t))}{t} = \sum_{n = k}^{\infty}\frac{A_n^{(k)}t^{n-1}}{\psi(t)^n}z^n = \sum_{n = k}^{\infty}\P(T_t \in \mathcal{G}_n^{(k)})z^n.
\end{align*}

\begin{lem} For any $\psi(z) = \sum_{n = 0}^{\infty}b_n z^n \in \K$ and any integer $k \geq 1$, the analytic function $g_k(z)$ verifies the recurrence relation
	\begin{align}\label{eq: recurrenceg_k}
		g_k(z) = z(\psi(g_{k-1}(z))-b_0),
	\end{align}
	where $g_0(z) = g(z)$. Here $g(z)$ denotes the solution of Lagrange's equation with data $\psi \in \K$.
\end{lem}
\begin{proof}
	We have that
	\begin{align*}
		\textsc{coeff}_{[n]}\left[z(\psi(g_{k-1}(z))-b_0)\right] &= \textsc{coeff}_{n-1}\left[\psi(g_{k-1}(z))-b_0\right] \\ &= \sum_{j = 1}^{n-1}b_j \textsc{coeff}_{n-1}[g_{k-1}(z)^j],
	\end{align*}
	therefore
	\begin{align*}
		\textsc{coeff}_{[n]}\left[z(\psi(g_{k-1}(z))-b_0)\right] = \sum_{j = 1}^{n-1}b_j \sum_{l_1+l_2+\dots+l_j = n-1}A_{l_1}^{(k-1)}A_{l_2}^{(k-1)}\dots A_{l_j}^{(k-1)}.
	\end{align*}
	
	For the family of rooted plane trees with $\partial \geq k$, that is, $\mathcal{G}^{(k)}$, we have:
	
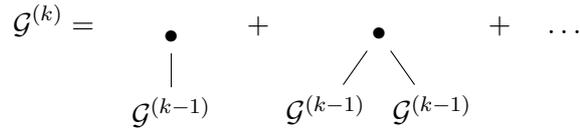
\begin{figure}[H]
		\begin{equation*}
			\mathcal{G}^{(k)} =  \ \ \treee{\mathcal{G}^{(k-1)}} \ \ + \treeC{\mathcal{G}^{(k-1)}} + \text{ \hspace{.2 cm}} \dots
		\end{equation*}
	\end{figure}
	In this diagram we fix a root and then we glue a tree with $\partial \geq k-1$, another root and then we glue 2 trees with $\partial \geq k-1$, and so on. By means of this process we build all the rooted plane trees with $\partial \geq k$.
	
	If we restrict our attention to the rooted plane trees with $\partial \geq k$ and $n$ nodes, the previous diagram should stop at certain point which depends on the value of $k$. For the rooted plane trees with $\partial \geq k$ and $n$ nodes, we allocate $n-1$ nodes in the trees which descend from the root: by means of this process we enumerate the class $\mathcal{G}^{(k)}_{n}$.
	
	If we glue $j$ rooted trees with $\partial \geq k-1$ to a root, then, when calculating the weights, we should weight the root using the coefficient $b_j$. The previous reasoning tell us that
	\begin{align*}
		\textsc{coeff}_{[n]}\left[z(\psi(g_{k-1}(z))-b_0)\right]  = \sum_{j = 1}^{n-1}b_j \sum_{l_1+l_2+\dots+l_j = n-1}A_{l_1}^{(k-1)}A_{l_2}^{(k-1)}\dots A_{l_j}^{(k-1)} =  A_n^{(k)}.
	\end{align*}
	
	See \cite{Flajolet} for similar arguments applied to combinatorial classes of rooted labeled trees or more generally to simple varieties of trees.
\end{proof}

Now, we give an asymptotic formula for the probability that a Galton-Watson process $T_t$, conditioned to have total progeny equal to $n$, has distance to the border $\partial(T_t) \geq k$, as the number of nodes $n \rightarrow \infty$.

\begin{theo}\label{thm: mainthmprobabilistic} Fix an integer $k \geq 0$ and a function $\psi \in \K^{\star}$ with $Q_{\psi} \geq 1$, then, for all $t \in (0,R_{\psi})$, we have
	\begin{align*}
		\lim_{\substack{n \to \infty; \\ n\in \mathbb{N}_{1,Q_\psi}}}\P(\partial(T_t) \geq k \, | \, \#(T_t) = n)  = \lim_{\substack{n \to \infty; \\ n\in \mathbb{N}_{1,Q_\psi}}}\frac{A_n^{(k)}}{A_n} = \frac{\partial G_k}{\partial w}(\rho,\tau).
	\end{align*}
\end{theo}

The proof of this Theorem follows combining Lemma \ref{lemma:uniform_galton} with Theorem \ref{thm:mainthmlagrange}. This result also appears in \cite{Protection1-mean}, the proof in there uses singularity analysis.

\begin{remark}
	For values of $t$ in the interval $(\tau, R_{\psi})$, we can also study the Galton-Watson process $T_t$ upon extinction. Take $a \in \mathcal{G}$, this is a finite rooted plane tree. Suppose that $a$ has exactly $n \geq 1$ nodes, then we have
	\begin{align}\label{eq: prob_extinction_conditionated}
		\P(T_t = a \, | \, \text{extinction})= \frac{\omega_{\psi}(a)t^{n-1}}{\psi(t)^n}\frac{1}{q(t)}.
	\end{align}
	
	Formula (\ref{eq: prob_extinction_conditionated}) is valid for all $t \in (0,R_{\psi})$, and, of course, we have
	\begin{align*}
		\P(T_t = a) = 	\P(T_t = a \, | \, \text{extinction})= \frac{\omega_{\psi}(a)t^{n-1}}{\psi(t)^n}, \quad \text{ for all } t \in (0,\tau].
	\end{align*}
	Here we use that $q(t) = 1$, for all $t \in [0,\tau]$.
	
	For each $t \in (0,R_{\psi})$ we denote $\tilde{T}_t$, the Galton-Watson process $T_t$ conditioned upon extinction. If we condition $\tilde{T}_t$ to have total progeny $n \geq 1$, then we find that
	\begin{align*}
		\P(\tilde{T}_t \in \mathcal{R} \,  |  \, \#(\tilde{T}_t) = n) = \frac{V_n}{U_n},
	\end{align*}
	for all $n \in \mathbb{N}_{1,Q_{\psi}}$.
	
	From here, and depending on the restrictions we impose to the class $\mathcal{R}$, we can find an asymptotic formula for this probability, as the number of nodes $n \rightarrow \infty$.
\finremark	
\end{remark}

\section{Applications}\label{sec: applications}

In this section we apply Theorem \ref{thm: mainthmprobabilistic} to some specific classes of rooted trees. By means of this result, we give asymptotic formulas for the rooted labeled Cayley trees (Poisson offspring distribution), rooted plane trees (geometric offspring distribution) and rooted binary plane trees (Bernoulli $\{0,2\}$ offspring distribution), with distance to the border bigger or equal than $k \geq 0$, as the number of nodes $n$ escapes to infinity.

For the study of the height we refer to \cite{Renyi}, \cite{de Bruijn}, and \cite{FlajoletOdlyzko} for the Cayley, plane and binary case respectively. 

\subsection{Rooted labeled Cayley trees} The class of rooted labeled Cayley trees with $n$ nodes $\mathcal{T}_n$ is a family of rooted trees where the descendants of each node are not ordered. This is a family of labeled trees, that is, for each rooted labeled Cayley tree with $n$ nodes there is a bijective correspondence between the set of nodes and the set of labels $\{1,2,\dots,n\}$.

The main property of this family of trees is that the trees in this class are not embedded in the plane. For the specific details about this family of trees we refer, for instance, to \cite[pp. 126-127]{Flajolet}.

The rooted labeled Cayley trees form a simple variety of trees, see, for instance \cite[pp. 126-127]{Flajolet}. A consequence of this fact is that the number of rooted labeled Cayley trees with $n$ nodes might be counted using Lagrange's equation with data $\psi(z) = e^z$.

The Cayley tree function $T(z)$, the exponential generating function for the counting sequence of this family of trees, satisfies Lagrange's equation
\begin{align}\label{eq: Lagrange-Cayley}
	T(z) = ze^{T(z)}, \quad \text{ for all } |z| \leq \frac{1}{e}.
\end{align}

Applying Lagrange's inversion formula, we find that
\begin{align}\label{eq: Cayley tree function}
	T(z) = \sum_{n = 1}^{\infty}\frac{n^{n-1}}{n!}z^n, \quad \text{ for all } |z| \leq \frac{1}{e}.
\end{align}

We want to apply Theorem \ref{thm: mainthmprobabilistic} to this particular family of rooted labeled trees. Each of the elements of the Khinchin family associated to $\psi(z) = e^z$, that we denote $(Y_t)_{t \in [0,\infty)}$, follows a Poisson distribution with parameter $t \in [0,\infty)$. See \cite{K_uno}, \cite{K_exponentials} and \cite{K_dos} for more examples of Khinchin families.

The mean function is given by
\begin{align*}
	m_{\psi}(t) = \frac{t\psi^{\prime}(t)}{\psi(t)}= t, \quad \text{ for all } t \in [0,\infty),
\end{align*}
in fact the equation $m_{\psi}(t) = 1$ has solution $\tau = 1.$

In this case the bivariate function $G(z,w)$ is given by
\begin{align*}
	G(z,w) = z(e^{w}-1),
\end{align*}
for all $|z| \leq 1/e$ and $|w| \leq 1$.

For any pair of integers $k \geq 0$ and $n \geq 1$, we denote $t_n^{(k)}$ the counting sequence for the rooted labeled Cayley trees with $n$ nodes and $\partial \geq k$. Besides, we denote $t_n$ the counting sequence for the rooted labeled Cayley trees with $n$ nodes. Formula (\ref{eq: Cayley tree function}) tell us that $t_n = n^{n-1}$, this is Cayley's formula for the enumeration of the rooted labeled Cayley trees. 

Now, applying Lemma \ref{lemma:uniform_galton}, we find that for any integer $k \geq 0$, and as $n \rightarrow \infty$, the proportion that trees in $\mathcal{T}_n^{(k)}$ do occupy in $\mathcal{T}_n$ is given by the limit
\begin{align*}
	\lim_{n \rightarrow \infty}\frac{\, \, \, \, t_n^{(k)}}{t_n} = \frac{\partial G_k}{\partial w}(1/e,1)\,.
\end{align*}

This asymptotic formula gives the answer to a question posed in \cite{Ara-Fer}, for an alternative proof of this result when $k = 2$ and $k = 3$, see \cite[pp. 309-312]{Ara-Fer}.

We denote by $c_k \triangleq \lim_{n \rightarrow \infty}t_n^{(k)}/t_n$. Applying the chain rule to $G_k(z,w)$ we obtain the following recurrence relation:
\begin{align}\label{eq: recucrence_k_cayley}
	c_k =\frac{1}{e}e^{G_{k-1}(1/e,1)}c_{k-1}, \quad \text{ for all } k \geq 1,
\end{align}
with $c_0 = 1$.

Let's study some particular instances of this formula.

\begin{itemize}
	\item[\scalebox{0.6}{$\blacksquare$}] The case $k =2$. Using the recurrence relation (\ref{eq: recucrence_k_cayley}), we find that
	\begin{align}\label{eq: c_2 Cayley}
		c_2 = \frac{1}{e}e^{G(1/e,1)}c_0 = \frac{1}{e}e^{1-1/e} = e^{-1/e}.
	\end{align}
	This result appears in \cite[p. 312]{Ara-Fer}.
	\vspace{.2 cm}
	
	\item[\scalebox{0.6}{$\blacksquare$}] The case $k = 3$. Using again our recurrence relation, see equation (\ref{eq: recucrence_k_cayley}), we find that
	\begin{align*}
		c_3 = \frac{1}{e}e^{G_2(1/e,1)}c_2.
	\end{align*}
	
	We have
	\begin{align*}
		G_2(1/e,1) = G(1/e,G(1/e,1)) = G(1/e,1-1/e) = \frac{1}{e}(e^{1-1/e}-1),
	\end{align*}
	and $c_2 = e^{-1/e}$, see equation (\ref{eq: c_2 Cayley}), therefore
	\begin{align*}
		c_3 = \frac{1}{e}e^{-1/e}e^{(e^{1-1/e}-1)/e}.
	\end{align*}
	This is Theorem 3.3, in \cite[p. 309]{Ara-Fer}.
	\vspace{.2 cm}
	
	\item[\scalebox{0.6}{$\blacksquare$}] The case $k = 4$. In this case we have
	\begin{align*}
		c_4 = \frac{1}{e}\exp\left(\frac{1}{e}(e^{(e^{1-1/e}-1)/e}-1)\right)\frac{1}{e}e^{-1/e}e^{(e^{1-1/e}-1)/e}.
	\end{align*}
\end{itemize}

In general, for any $k \geq 4$, we might continue iterating the recurrence relation (\ref{eq: recucrence_k_cayley}) and therefore find any particular value of $c_k$. This means that, for any $k \geq 1$, the constant $c_k$ is computable by iteration.

The recurrence relation (\ref{eq: recucrence_k_cayley}) is the answer to a question posed in \cite[p. 312]{Ara-Fer}. Some of these quantities appear in \cite{Protection1-mean}.

\subsection{Rooted plane trees} The rooted plane trees form a simple variety of trees, that is, we can count these trees by means of Lagrange's equation with data $\psi(z) = 1/(1-z)$.

The generating function for this family of rooted trees, that we denote $P(z)$, verifies Lagrange's equation with data $\psi$:
\begin{align}\label{eq: Lagrange-plane}
	P(z) = \frac{z}{1-P(z)}, \quad \text{ for all } |z| \leq \frac{1}{4}.
\end{align}

Applying Lagrange's inversion formula, we find that
\begin{align}\label{eq: generating_plane}
	P(z) = \sum_{n = 1}^{\infty}C_{n-1}z^n, \quad \text{ for all } |z| \leq \frac{1}{4},
\end{align}
where $C_n$ denotes the $n$th Catalan number.

We want to apply Theorem \ref{thm: mainthmprobabilistic} to the family of rooted plane trees. Each of the elements of the Khinchin family associated to $\psi(z) = 1/(1-z)$, that we denote $(Y_t)_{t \in [0,1)}$, follows a geometric distribution with parameter $1-t$. See \cite{K_uno}, \cite{K_exponentials} and \cite{K_dos} for examples of Khinchin families and for more details about this specific family of random variables.

The mean function of this Khinchin family is
\begin{align*}
	m_{\psi}(t) = \frac{t}{1-t}, \quad \text{ for all } t \in [0,1),
\end{align*}
in fact the equation $m_{\psi}(t) = 1$ has solution $\tau = 1/2$.

For $\psi(z) = 1/(1-z)$ we have that
\begin{align*}
	G(z,w) = z\left(\frac{1}{1-w}-1\right) = \frac{zw}{1-w},
\end{align*}
for all $|z| \leq 1/4$ and $|w| \leq 1/2$.

In this particular case it is possible to find a closed expression for the $k$-th iterate of $G(z,w)$, this is the content of the following Lemma.
\begin{lem} For any $k \geq 1$, the $k$-th iterate of $G(z,w)$ verifies that
	\begin{align}\label{eq: formula_plane_iterate}
		G_k(z,w) = \frac{z^k w}{1-\frac{1-z^k}{1-z}w}, \quad \text{ for all } |z| \leq 1/2, |w| \leq 1/4.
	\end{align}
\end{lem}

This Lemma follows by induction on $k \geq 1$. Applying this result we will be able to provide an explicit expression for the asymptotic formula given by Theorem \ref{thm: mainthmprobabilistic}. This is a particularity of this example.

We cannot find a closed expression, in the spirit of equation (\ref{eq: formula_plane_iterate}), for the family of rooted labeled Cayley trees. See the recurrence relation (\ref{eq: recucrence_k_cayley}).

For any pair of integers $k \geq 0$ and $n \geq 1$ we denote $P_n^{(k)}$ the counting sequence for the rooted plane trees with $n$ nodes and $\partial \geq k$. Besides, we denote $P_n$ the counting sequence for the rooted plane trees with $n$ nodes. Equation (\ref{eq: generating_plane}) gives that $P_n = C_{n-1}$. This formula follows applying Lagrange's inversion formula to equation (\ref{eq: Lagrange-plane}).

Finally, applying Theorem \ref{thm: mainthmprobabilistic} to the family of rooted plane trees, we find the following result
\begin{theo} For any integer $k \geq 0$, and as $n \rightarrow \infty$, the proportion that trees in $\mathcal{G}_n^{(k)}$ do occupy in $\mathcal{G}_n$ is given by the limit
	\begin{align*}
		\lim_{n \rightarrow \infty}\frac{\, \, \, \,  P_n^{(k)}}{P_n} = \frac{\partial G_k}{\partial w}(1/4,1/2) = 9\frac{4^k}{(2+4^k)^2}\,.
	\end{align*}
\end{theo}
This result also appears in \cite{Protection-Gisang}.

\subsection{Rooted binary plane trees} Now we study the rooted binary plane trees, these are rooted plane trees where each node can only have 0 or 2 descendants. We can count these trees using Lagrange's equation with data $\psi(z) = 1+z^2$. 

The generating function for the rooted binary plane trees is
\begin{align*}
	B(z) = \sum_{n = 0}^{\infty}C_{n}z^{2n+1}, \quad \text{ for all } |z| \leq 1/2,
\end{align*}
where $C_n$ denotes the $n$th Catalan number. The generating function for the rooted binary trees with $\partial \geq k$ is
\begin{align*}
	B_k(z) = G_k(z,B(z)) = \sum_{n = 1}^{\infty}B_n^{(k)}z^{2n+1}, \quad |z| \leq 1/2.
\end{align*}

Now we write
\begin{align*}
	G(z,w) = zw^2, \quad \text{ for all } |z| \leq 1/2, |w| \leq 1. 
\end{align*}

We can find a closed expression for the iterate $G_k(z,w)$. This is the content of the following lemma. 

\begin{lem} For any $k \geq 0$, we have
	\begin{align*}
		G_k(z,w) = z^{2^k-1}w^{2^k}, \quad \text{ for all } |z| \leq 1/2, |w| \leq 1.
	\end{align*}
\end{lem}

Using the previous Lemma we conclude that the partial derivatives of $G_k(z,w)$ are given by
\begin{align}\label{partial-binary}
	\frac{\partial G_{k}}{\partial z}(z,w) = (2^{k}-1)z^{2^{k}-1}w^{2^{k}}, \quad \frac{\partial G_{k}}{\partial w}(z,w) = 2^{k}z^{2^{k}-1}w^{2^{k}-1}.
\end{align}

In this particular case we also have an explicit formula for the generating function of the rooted binary plane trees with $\partial_{root} \geq k$.

\begin{coro} For $k \geq 0$ we have
	\begin{align}\label{generating-k-binary}
		B_k(z) = G_k(z,B(z)) = z^{2^{k}-1}B(z)^{2^{k}}, \quad \text{ for all } \text{\vspace{.1 cm} } |z| \leq 1/2.
	\end{align}
\end{coro}

Using this representation, we can give an exact formula for the coefficients of $B_k(z)$. Later on, using this concrete formula, we give an asymptotic formula for the coefficients $B_n^{(k)}$. 

For any $n = 2m+1$, for $m$ large enough, we have
\begin{align}\label{eq: asymptotic A_n^{(k)}-binaryplane}
	\begin{split}
		\textsc{coeff}_{[2m+1]}\left[B_k(z)\right] &= \textsc{coeff}_{[2m-2^k+2]}\left[B(z)^{2^k}\right] \\ &= \frac{2^k}{2m-2^k+2}\textsc{coeff}_{[2m-2^{k+2}+2]}\left[(1+z^2)^{2m-2^k+2}\right] \\
		&= \frac{2^k}{2m-2^k+2}{2m-2^k+2 \choose m-2^{k}+1} \sim 2^{k-2^{k}+1} \frac{4^{m}}{\sqrt{\pi}m^{3/2}}, \quad \text{ as } m \rightarrow \infty.
	\end{split}
\end{align}

The asymptotic formula (\ref{eq:formula OtterMeinMoon}) gives that
\begin{align}\label{eq: asymptotic A_n-binaryplane}
	\textsc{coeff}_{[2m+1]}\left[B(z)\right] = C_{m} = \frac{1}{m}{2m \choose m} \sim \frac{4^{m}}{\sqrt{\pi}m^{3/2}}, \quad \text{ as } m \rightarrow \infty.
\end{align}

The following Theorem follows combining formulas (\ref{eq: asymptotic A_n^{(k)}-binaryplane}) and (\ref{eq: asymptotic A_n-binaryplane}). This result is also a direct consequence of Theorem \ref{thm: mainthmprobabilistic}.

\begin{theor} For any integer $k \geq 0$ and $n \geq 1$ odd we have
	\begin{align*}
		\lim_{\substack{n \to \infty; \\ n\in \mathbb{N}_{1,2}}}\frac{\, \, \, \, B_n^{(k)}}{B_n} = \frac{\partial G_k}{\partial w}(1/2,1) = 2^{k-2^k+1}\,.
	\end{align*}
\end{theor}

\section{Mean number of nodes with $\partial_v \geq k$ in a Cayley tree}\label{sec:non-rooted}

For the family of unrooted labeled Cayley trees, and following a circle of ideas appearing in \cite{Ara-Fer}, we study the mean number of nodes which are at distance to the border bigger or equal than $k$, as the number of nodes $n \rightarrow \infty$.

Recapitulating, for a given tree $T$, and for any integer $k \geq 0$, a node $v \in \text{nodes}(T)$ is at distance to the border bigger or equal than $k$ if 
$$\partial_{v} \triangleq \min_{u \in \text{leaves}(T)}d(u,v) \geq k.$$

We denote by $\mathcal{U}_n$ the class of unrooted labeled Cayley trees. The nodes of each of the trees in this family can be distinguished by its labels, therefore, for each tree in $\mathcal{U}_n$ there are $n$ rooted labeled Cayley trees. In particular, the number of unrooted labeled Cayley trees is given by the sequence $U_n = n^{n-2}$. 

Let $X_{n,k}$ be the random variable that, for any tree $T \in \mathcal{U}_n$, counts the number of nodes $v$ with distance to the border $\partial_{v} \geq  k$, then we have
\begin{align*}
	X_{n,k}(T) = \sum_{u \in \text{nodes}(T)}\textbf{1}_{\{\partial_{u} \geq k\}}.
\end{align*}

Following \cite{Ara-Fer}, we consider a 0-1 matrix $M$ of size $n \times n^{n-2}$. We label the columns of this matrix with the collection of trees in $\mathcal{U}_n = \{t_1,t_2,\dots\}$, and then the rows with the labels of each node, that is, $\{1,2,\dots,n\}$. We write $1$ in the $(j,t_j)$ entry if the node $j$ is at distance to the border bigger or equal than $k$ and 0 otherwise.

By adding all the entries of $M$ and then dividing by $n^{n-1}$ we obtain the mean value of the random variable $X_{n,k}$, that is,
\begin{align*}
	\mathbb{E}_n(X_{n,k}) = \frac{1}{n^{n-1}}\sum_{T \in \, \mathcal{U}_n}X_{n,k}(T).
\end{align*}

For any integer $k \geq 0$, and using the notation
\begin{align*}
	c_k = \lim_{n \rightarrow \infty}t_n^{(k)}/t_n\,,
\end{align*}
we obtain the following result.

\begin{theor}\label{thm:labeled_meannumber} As $n \rightarrow \infty$, the expectation of the proportion of nodes in a labeled Cayley tree with $n$ nodes that are at distance bigger or equal than $k$ from any leaf tends to $c_k$, that is, 
	\begin{align*}
		\lim_{n \rightarrow \infty}\frac{1}{n}\mathbb{E}_n(X_{n,k}) = c_k\,.
	\end{align*}
	Equivalently,
	\begin{align*}
		\mathbb{E}_n(X_{n,k}) \sim c_k \cdot n, \quad \text{ as } n \rightarrow \infty\,. 
	\end{align*}
\end{theor}

This same result is true for any subclass of rooted labeled Cayley trees with exponential generating function given by Lagrange's equation with data $\psi(z) = \sum_{j \geq 0}(b_j/j!)z^j \in \K^{\star}$, where $b_j \in \{0,1\}$, for all $j \in \{0,1,\dots\}$, and total size $n \in \mathbb{N}_{1,Q_{\psi}}$.

Finally, for the unrooted labeled unary Cayley trees, that is, the family of trees given by $\psi(z) = 1+z \in \K$, we have
\begin{align*}
	\mathbb{E}_n(X_{n,k}) \sim n, \quad \text{ as } n \rightarrow \infty\,.
\end{align*}

Some particular cases of Theorem \ref{thm:labeled_meannumber} appear in \cite{Ara-Fer}, there the authors prove the previous result for nodes $v$ with $\partial_{v} \geq 3$, see Theorem 3.7 in \cite[p. 314]{Ara-Fer}.

\end{document}